\documentclass[12pt, one sided]{amsart}

\usepackage{amsmath}
\usepackage{amssymb}
\usepackage{amsthm}

\theoremstyle{plain}
\newtheorem{thm}{Theorem}[section]
\newtheorem{prop}[thm]{Proposition}
\newtheorem{lem}[thm]{Lemma}
\newtheorem{cor}[thm]{Corollary}

\newtheorem{prob}[thm]{Problem}
\newtheorem{fact}[thm]{Fact}

\theoremstyle{definition}
\newtheorem{defn}[thm]{Definition}
\newtheorem{eg}[thm]{Example}
\newtheorem{set}[thm]{Setting}

\theoremstyle{remark}
\newtheorem{rem}[thm]{Remark}
\newtheorem{claim}[thm]{Claim}

\DeclareMathOperator{\mult}{mult}
\DeclareMathOperator{\Pic}{Pic}
\DeclareMathOperator{\Sing}{Sing}
\DeclareMathOperator{\Fix}{Fix}
\DeclareMathOperator{\Bs}{Bs}
\DeclareMathOperator{\Neg}{Neg}
\DeclareMathOperator{\Nef}{Nef}

\DeclareMathOperator{\gen}{gen}
\DeclareMathOperator{\sm}{sm}

\newcommand{\NEbar}{\mathop{\overline{\mathrm{NE}}}}

\begin{document}

\title
[Seshadri constants on rational surfaces ]
{Seshadri constants on rational surfaces with anticanonical pencils }
\subjclass[2010]{Primary 14C20, 14J26; Secondary 14D06}
\keywords{Seshadri constants, rational surfaces, log del Pezzo surfaces, T-singularities}

\author{Taro Sano}
\address{Mathematics Institute, 
Zeeman Building, 
University of Warwick, 
Coventry, CV4 7AL, UK}
\email{T.Sano@warwick.ac.uk}

\maketitle

\begin{abstract}
We study a Seshadri constant at a general point on a rational surface whose anticanonical linear system contains a pencil.
First, we describe a Seshadri constant of an ample line bundle 
 on such a rational surface explicitly by the numerical data of the ample line bundle. 
Secondly, we classify log del Pezzo surfaces which are special in terms of the Seshadri constants 
of the anticanonical divisors 
when the anticanonical degree is between 4 and 9. 
\end{abstract}

\tableofcontents

 \section{Introduction}

We consider projective varieties over an algebraically closed field $\mathbb{K}$  of characteristic zero 
throughout this note.

The following criterion for ampleness is called Seshadri's criterion (\cite{Laz} Theorem 1.4.13).

\begin{fact}\label{sescri}
Let $X$ be a projective variety and $L$ a line bundle on $X$. 
Then $L$ is ample if and only if there exists a positive number $\varepsilon > 0$ 
such that 
\[
\frac{L\cdot C}{\mult_x(C)} \geq \varepsilon  
\]
for every point $x \in X$ and every irreducible curve $C \subset X$ passing through $x$.
\end{fact}

Demailly \cite{Dem} defined the following number which is related to the above fact. 

\begin{defn}\label{sesdef}
Let $X$ be a projective variety, $L$ an ample line bundle on $X$ and 
 $x \in X$. 
We define the  {\em Seshadri constant} of $L$ at $x$ to be 
\[
\varepsilon(L,x) := \inf \frac{L\cdot C}{\mult_x(C)},  
\]
where the infimum is taken over all irreducible reduced curves $C$ in $X$ passing through $x$.
\end{defn}

\begin{rem}
It is well-known that $\varepsilon(L,x) = \max\{s \in \mathbb{R};\mu_x^{\ast}(L)-sE_x \ \text{is nef} \}$,  
 where $\mu_x: \widetilde{X}(x) \rightarrow X$ is the blow-up at $x$ and $E_x:= \mu_x^{-1}(x)$ is the exceptional divisor .
\end{rem} 

\begin{rem}\label{reducible}
Even if we take the infimum over all $1$-dimensional cycles $C$ on $X$ in Definition \ref{sesdef}, we 
get the same constant as $\varepsilon(L,x)$, that is, we do not need to assume that $C$ is irreducible and reduced in Definition \ref{sesdef}. 
This follows from the fact that $\mult_x(C_1) + \mult_x(C_2) = \mult_x(C_1+C_2)$ for curves $C_1,C_2$ 
on $X$ and that, for any positive numbers $d_1, d_2, m_1, m_2$, we have 
\[
\min \left\{\frac{d_1}{m_1}, \frac{d_1}{m_2} \right\} \le \frac{d_1+d_2}{m_1+m_2}.
\]
\end{rem}

\begin{rem}
We have $\varepsilon(kL,x) = k \varepsilon(L,x)$ for all $k \in \mathbb{Z}_{>0}$. 
Since $\varepsilon(L,x)$ depends only on the numerical class of $L$, we can define $\varepsilon(L,x)$ for a nef $\mathbb{R}$-divisor $L$.  
\end{rem}

\begin{rem}[Definition of $\varepsilon_{\gen}(L)$]
Let $\varepsilon_L$ be the Seshadri function for $L$ defined by 
\[
\varepsilon_L: X \rightarrow \mathbb{R} ; x \mapsto \varepsilon(L,x). 
\]
Then $\varepsilon_L$ takes a constant value at very general points $x$. We denote this value by $\varepsilon_{\text{gen}}(L)$ 
and call it the {\em Seshadri constant of $L$ at a very general point}.
\end{rem}

A Seshadri constant measures the local positivity of an ample line bundle on a projective variety. 
It appears in various situations. For example, it is related to generation of jets of adjoint bundles, 
Castelnuovo-Mumford regularity, and the existence of cscK metrics. We refer to \cite{primer}  for a survey.  

Unfortunately, a Seshadri constant is hard to compute, even for a surface. 
 Since the definition involves every point and every curve, for a general variety, we have the problem of calculating
 with the totality of all curves on $X$. 
Whether a curve contributes to the Seshadri constant depends on the individual curve, 
not just on its linear or algebraic or rational equivalence class.
There are many papers dealing with explicit computations of Seshadri constants, for example, \cite{B1,B2,BSch,BS,Br,G}. 

\vspace{5mm}

In this note, we give an explicit formula for a Seshadri constant on a rational surface $X$ with $\dim |{-}K_X| \ge 1$ as follows. 
(See Theorem \ref{main} for the precise statement.)

\begin{thm}\label{mainintro}
Let $X$ be a smooth rational surface such that $h^0(X,-K_X) \ge 2$.  
Set $r := 9- K_X^2$. 
Let $L$ be an ample line bundle on $X$. Then, for a general point $x \in X$, we have 
\[
 \varepsilon(L,x) =\begin{cases}
-K_X \cdot L & \text{ if } r=8, 9 \text{ and } {-}K_{\tilde{X}(x)} \cdot (\mu_x^*L- \sqrt{L^2} E_x) \le 0 \\
\min \{ A_{M_{L}}(L), -K_X \cdot L \} & \text{otherwise, } 
\end{cases}
\]
where  $\mu_x \colon \tilde{X}(x) \rightarrow X$ is a blow-up at $x$, $E_x := \mu_x^{-1}(x)$ and  $A_{M_L}(L)$ is the contribution of finitely many $-1$-curves on $\tilde{X}(x)$. We can calculate $A_{M_L}(L)$ explicitly if $r \le 9$. 
\end{thm}

This is a generalisation of the formula of Broustet \cite{Br} which describes the Seshadri constant of an anticanonical divisor of 
a del Pezzo surface. In our formula, we treat any polarisation on a rational surface $X$ such that $h^0(X,-K_X) \ge 2$. 
 For example, we can treat $X$ such that $r \le 8$ or a rational elliptic surface with a section. 
We can calculate the Seshadri constant at $x$ in $X$ from a finite collection of numerical data 
that is known a priori:  the Seshadri constant is achieved by a curve 
that is either the push-forward of a $-1$-curve of low degree on the
blow-up $\tilde{X}(x) \rightarrow X$, or by an anticanonical curve through $x$.

\vspace{5mm}

By using Theorem \ref{mainintro}, we study a Seshadri constant on a log del Pezzo surface and a relation with its singularities. 
Seshadri constants vary lower semi-continuously in a flat family (Proposition \ref{lsc}). 
Hence  a variety with a small Seshadri constant can be considered to be ``special''. 
Nakamaye studied a Seshadri constant on an abelian variety and 
characterised an abelian variety of product type via a Seshadri constant in \cite{N}.    
Inspired by this result, we study a Seshadri constant of an anticanonical divisor of a log del Pezzo surface and 
determine those with large Seshadri constants.  
The result is as follows. 
 
\begin{thm}\label{ldpclassifintro}
Let $X$ be a log del Pezzo surface. 
\begin{enumerate}

\item[(i)] Suppose that $K_X^2 =4,5,6,7,8$. 

Then $\varepsilon_{\gen}(-K_X) = 2$ 
if and only if 
 $X$ has only canonical singularities or $ X $ is one of the $7$ types of the surfaces 
$Z_i (i=1,\ldots,7)$ which are defined in Remark \ref{speldp}. That is, for $X$ with non-canonical singularities which is not 
$Z_i$, we have $\varepsilon_{\gen}(-K_X) <2$.   
\item[(ii)] Suppose that $K_X^2 = 9$. 

Then $\varepsilon_{\gen}(-K_X) = 3$ if and only if 
 $X \simeq \mathbb{P}^2$.  
That is, for $X$ which is not $\mathbb{P}^2$, we have $\varepsilon_{\gen}(-K_X)<3$. 
\end{enumerate}
\end{thm}
  
The above $Z_i$ has a cyclic quotient singularity of type $1/4(1,1)$ which is not canonical and several $A_1$-singularities. 
However a $1/4(1,1)$-singularity is mild and special among general quotient singularities.   
Hence Theorem \ref{ldpclassifintro} means that ``special'' log del Pezzo surfaces $X$ are those with 
non-canonical singularities which are not $Z_i$ when $4 \le K_X^2 \le 8$.

\subsection*{The contents of this note}
We summarize the contents of this note. 

In Section \ref{state},  Theorem \ref{main} states an explicit formula for
a Seshadri constant $\varepsilon_{\gen}(L)$ of an ample line bundle $L$ on a rational surface $X$ 
with a birational morphism $X \rightarrow \mathbb{P}^2$ such that $h^0(X, -K_X) \ge 2$.

In Section \ref{prfmain}, we prove Theorem \ref{main}. 
In \ref{reduction1}, we prove a lemma about negative curves on an anticanonical rational surface and 
restrict curves which are necessary for computing $\varepsilon(L,x)$. 
Note that we pay attention to curves on $\tilde{X}(x)$ rather than curves on $X$. 
In Lemma \ref{bound} of the subsection \ref{boundsub}, we establish finiteness of degrees of curves $\tilde{C}$ on $\tilde{X}(x)$  
which are necessary for the calculation. Note that there are infinitely many $-1$-curves on $\tilde{X}(x)$ in general if $r \ge 8$. 
It turns out that we can describe the bound explicitly when $r \le 9$.   
In \ref{finalred}, we show that, in order to compute $\varepsilon(L,x)$ for a general point $x$, 
it is enough to consider a curve $\tilde{C} \subset \tilde{X}(x)$ which is either  
a $-1$-curve or an  anticanonical curve. 
In \ref{hirzsection}, we compute $\varepsilon_{\gen}(L)$ on $X$ with a birational morphism to a Hirzebruch surface $\mathbb{F}_n$. 
In \ref{exsection}, we give an example of an explicit computation of $\varepsilon_{\gen}(L)$ on a del Pezzo surface $X$. 

In Section \ref{ldpstate}, we state the result  
on $\varepsilon_{\gen}(-K_X)$ of a log del Pezzo surface $X$ in Theorem \ref{ldpclassif}. 

In Section \ref{ldp}, we prove Theorem \ref{ldpclassif}.  
We calculate $\varepsilon_{\gen}(-K_X)$ by calculating $\varepsilon_{\gen}(\nu^{\ast}(-K_X))$ 
on the minimal resolution $\nu: Y \rightarrow X$.
In almost all cases, we can show that $\varepsilon_{\gen}(-K_X) < 2$ by using the strict transform $\tilde{l} \subset Y$ of a certain line $l$ on 
$\mathbb{P}^2$.

\section{Notation}\label{notation}
Let $X$ be a smooth projective surface. 
Put $N^1(X) := \left( \Pic (X) \otimes_{\mathbb{Z}} \mathbb{R}\right) / {\equiv}$, where $\equiv$ means numerical equivalence.
Let $\NEbar (X) \subset N^1(X)$ be the closure of the convex cone generated by classes of effective curves 
 and $\Nef (X) \subset N^1(X)$ be the convex cone generated by classes of nef divisors. 
For $D \in \Pic X$, let $[D] \in N^1(X)$ be its numerical class.  
 A divisor class $D \in \Pic(X)$ is called a $-m$-{\it class} if $D^2 = -m,K_X \cdot D = m-2$,  
where $K_X$ is the canonical divisor of $X$. 
A $-m$-class $D$ is called a $-m$-{\it cycle} (resp. a $-m$-{\it curve}) if it is represented by an effective divisor
(resp. an irreducible divisor). 

For 
$(\alpha; \stackrel{k_1}{\overbrace{\beta_1, \ldots, \beta_1}}, \ldots, 
\stackrel{k_l}{\overbrace{\beta_l, \ldots, \beta_l}}) \in \mathbb{Z}^{1+N}$, where $N = \sum_{i=1}^l k_i$, 
we write $(\alpha; \beta_1^{k_1}, \ldots, \beta_l^{k_l})$ for short.

\section{An explicit formula on a blow-up of $\mathbb{P}^2$}\label{state}

In this section, we state a formula for the Seshadri constant of 
 an ample line bundle on a rational surface $X$ with a birational morphism $X \rightarrow \mathbb{P}^2$ such that $\dim |{-}K_X| \ge 1$.
We first fix the setting.

\begin{set}\label{setting}  
Let $X$ be a smooth rational surface such that $\dim |{-} K_X | \ge 1$ with a birational morphism $\mu \colon X \rightarrow \mathbb{P}^2$ which is a composition of blow-ups  
 \begin{equation}
 \mu \colon X=X_{r+1} \stackrel{\mu_r}{\rightarrow} X_r \rightarrow 
\cdots \stackrel{\mu_1}{\rightarrow} X_1 
  = \mathbb{P}^2, 
  \end{equation} 
  where $\mu_i \colon X_{i+1} \rightarrow X_i$
is the blow-up at $x_i \in X_i$. 
Let
 \begin{equation}\label{ampleform}
 L:= \mu^{\ast} \mathcal{O}_{\mathbb{P}^2}(a) - \sum_{i=1}^r b_i E_i 
 \end{equation}
 be an ample Cartier divisor on $X$, where $a,b_1,\ldots, b_r$ are integers and we set 
 \[
E_i:= (\mu_{i+1} \circ \cdots \circ \mu_r)^* \mu_i^{-1}(x_i)
\]
 for $i=1,\ldots,r$ which is a $-1$-cycle. 
 We assume that $L$ is primitive, that is $\gcd (a, b_1, \ldots, b_r) = 1$. 
 For $s \in \mathbb{Z}_{>0}$, set   
\[
\Phi_s := 
 \left\{(d;m_1,\ldots,m_s) \in \mathbb{Z}^{s+1}  \mid  
d^2 - \sum_{i=1}^s m_i^2 = -1, 3 d - \sum_{i=1}^s m_i = 1,
m_s \ge 1  \right\}. 
\]
We call an element of $\Phi_s$ a {\it $-1$-class}. 
\end{set}

\begin{rem}
 The condition $\dim |{-}K_X| \ge 1$ is satisfied if $r \le 8$ or if $X$ is a rational elliptic surface with a section, 
 for example. Note that the centers $x_i \in X_i$ can be infinitely near points. 
Also note that $|{-}K_{\widetilde{X}(x)}| \neq \emptyset$, where $\mu_x : \tilde{X}(x)\rightarrow X$ is a blow-up at $x$. 
\end{rem} 

Here is the explicit formula for the Seshadri constant $\varepsilon_{\gen}(L)$ of $L$ at a general point. 
The Seshadri constant is computed by an anticanonical curve on $X$ or a curve on $X$  whose strict transform on $\tilde{X}(x)$ is a $-1$-curve. 

\begin{thm}\label{main}
Let $X$ be a smooth rational surface and $L$ an ample line bundle on $X$ as in Setting \ref{setting}. Set \[
b_{r+1}:= \sqrt{\mathstrut L^2} = \sqrt{\mathstrut a^2 - \sum_{i=1}^r b_i^2 }. 
\]  
Then we have the following. 
\begin{enumerate}
\item[(i)] There exists a positive number $M_L$ which is determined by $L$ such that  
\[
\varepsilon_{\gen}(L) =\begin{cases}
-K_X \cdot L & \text{ if } r=8, \sum_{i=1}^9 \frac{b_i}{a} =3 \text{ or } r=9, \sum_{i=1}^{10} \frac{b_i}{a} \ge 3 \\
\min \{ A_{M_{L}}(L), -K_X \cdot L \} & \text{otherwise, } 
\end{cases}
\]
where 
\[
A_{M_L}(L):= \min_{(d;m_1,\ldots,m_{r+1}) \in \Phi_{r+1}, d \le M_L} 
\frac{d a - \sum_{i=1}^r m_i b_i}{m_{r+1}}.  
\]
\item[(ii)] If $r=8,9$ and  $\sum_{i=1}^{r+1} \frac{b_i}{a} < 3$, then $M_L$ is given by 
\[
M_L := \frac{2- c + \sqrt{c^2 +20c +4}}{2c},  
\] 
where we set $c:= 3- \sum_{i=1}^{r+1} \frac{b_i}{a}$.
\item[(iii)] If $\sum_{i=1}^{r+1} \frac{b_i}{a} =3$, then $M_L$ is given by 
\[
M_L := \frac{a^2+1}{2a} .
\]
\item[(iv)] There exists a nonempty Zariski open subset $U_L \subset X$ such that $\varepsilon_{\gen}(L) = \varepsilon(L,x)$ for all $x \in U_L$. 
\end{enumerate}
\end{thm}

\begin{rem}
We explain the description of $M_L$ in the cases which are not covered in (ii) and (iii) of Theorem \ref{main}. 

If $r \le 7$, we set $M_L =6$ since the maximum value of $d$ in $\Phi_{8}$ is $6$. In fact,  
$\Phi_{r+1}$ is a finite set if $r \le 7$. 

We do not give an explicit description of $M_L$ if $r \ge 10$. 
We can describe $M_L$ by a calculation which is more detailed than that in Lemma \ref{bound}. 
However, we omit this for simplicity. 
\end{rem}

\begin{rem} Suppose that $M_L$ is given explicitly as in the case $r=8$. 
 We can list up all the $-1$-classes in the finite set 
$\Phi_{r+1}(M_L):=\{(d;m_1, \ldots, m_{r+1}) \in \Phi_{r+1} \mid d \le M_L \}$ since $0 \le m_i < d$ for each $i$. 
So we can compute $\varepsilon_{\gen}(L)$ from finitely many data. We go over a computational example explicitly 
in Example \ref{compex}. 
\end{rem}

\begin{rem}
In Section \ref{hirzsection}, we explain how to compute $\varepsilon_{\gen}(L)$ on a rational surface $X$ such that $h^0(X,-K_X) \ge 2$ 
with a birational morphism $X \rightarrow \mathbb{F}_n$. 
\end{rem}

\section{Proof of Theorem \ref{main}} \label{prfmain}

\subsection{First reduction}\label{reduction1}

First, the following fact about the effective cone of an anticanonical rational surface 
is fundamental for us. 

\begin{prop}\label{reveal}  
Let $S$ be an smooth rational surface. Assume that $S$ is {\it anticanonical}, 
that is,  $|{-}K_S| \neq \emptyset$. Then we have the following.   
\item[(i)] We have 
\[
\NEbar (S) = \mathbb{R}_{\ge 0} [-K_S] + \sum_{\substack{
C \subset S \text{ an irreducible curve}, \\
C^2 < 0}} \mathbb{R}_{\ge 0} [C].
\]
\item[(ii)] Let $C \subset S$ be an irreducible reduced curve such that $C^2<0$. 
Then $C$ is either a $-1$-curve, a $-2$-curve or a fixed component of $|{-}K_S|$. 
\item[(iii)] Let $D$ be a Cartier divisor on $S$. Then $D$ is nef if and only if $D \cdot C \ge 0$ for all $C$ such that 
$C$ is either a $-1$-curve, a $-2$-curve or a fixed component 
of $|{-}K_S|$. 
\end{prop}

\begin{proof}
\noindent (i) See \cite[Lemma 4.1]{LH}. 

\noindent (ii) Suppose that $C$ is not a fixed component of $|{-}K_S|$. 
Then we have $-K_S \cdot C \ge 0$ 
and $2p_a(C) -2 = \left( K_S + C \right) \cdot C \le C^2 < 0$, where $p_a(C)$ is the arithmetic genus of $C$. 
Therefore we get $C \simeq \mathbb{P}^1$ and $C^2 = -1$ or $-2$.   

\noindent (iii) This follows from (i) and (ii). 
\end{proof}

We want to compute $\varepsilon(L,x)$ on a rational surface $X$  such that $\dim |{-}K_X| \ge 1$ at $x \in X$. 
We use the notations in Section \ref{state}. 
Recall that $\varepsilon(L,x) = \max \{s \in \mathbb{R} \mid \mu_x^{\ast}L - sE_x \text{ is nef}  \}$, where 
$\mu_x \colon \tilde{X}(x) \rightarrow X$ is the blow-up at $x$ and $E_x:= \mu_x^{-1}(x)$.  
In order to check the nefness of an $\mathbb{R}$-divisor $\mu_x^* L- s E_x$, it is enough to check   
\begin{equation}\label{reduce1}
(\mu_x^* L - s E_x) \cdot \tilde{C} \ge 0,  
\end{equation}
where $\tilde{C} \subset \tilde{X}(x)$ is either a $-1$-curve, a $-2$-curves, a fixed component of $|{-}K_{\tilde{X}(x)}|$ 
or an anticanonical curve on $\tilde{X}(x)$ since $|{-}K_{\tilde{X}(x)}| \neq \emptyset$ and 
we can use Lemma \ref{reveal}(iii) for $S = \tilde{X}(x)$.  
The problem is that there are infinitely many $-1$-curves and $-2$-curves on $\tilde{X}(x)$ in general. 
In the following, we restrict curves more precisely when $X$ has a birational morpshism to $\mathbb{P}^2$. 

Let $X, L$ be as in Setting \ref{setting}. 
For $(k;l):= (k; l_1,\ldots,l_{r+1}) \in \mathbb{Z}^{r+2}$, set  
\[
\mathcal{O}_{\tilde{X}(x)}(k; l) 
:= \mu_x^{\ast}\left( \mu^{\ast} \mathcal{O}_{\mathbb{P}^2}(k) - \sum_{i=1}^r l_i E_i \right) - l_{r+1} E_x. 
\]
Note that we have $\varepsilon(L,x) \le \sqrt{L^2}$ since $(\mu_x^*L- \epsilon(L,x) E_x)^2 \ge 0$. 
Hence it is enough to consider curves $\tilde{C} \subset \tilde{X}(x)$ such that 
\begin{equation}\label{necineq}
(\mu_x^* L - \sqrt{L^2}E_x) \cdot \tilde{C} \le 0. 
\end{equation}
For such a curve $\tilde{C} \subset \tilde{X}(x)$, let $d, m_1, \ldots, m_{r+1}$ be integers such that 
 \[
 \mathcal{O}_{\tilde{X}(x)}(\tilde{C}) = \mathcal{O}_{\tilde{X}(x)}(d; m_1, \ldots, m_{r+1}).
 \] 
  Then 
the inequality (\ref{necineq}) is equivalent to 
\begin{equation}\label{necineq2}
d a - \sum_{i=1}^{r+1} m_i b_i \le 0. 
\end{equation}

\begin{rem}\label{=3rem} 
We give a remark about the rationality of $\varepsilon(L,x)$. 

When $\sum_{i=1}^{r+1} \frac{b_i}{a} \neq 3$, in Proposition \ref{bound}, we show  
that there are only finitely many $-1$-classes and $-2$-classes $(d;m_1, \ldots, m_{r+1})$ which satisfy (\ref{necineq2}). 
Next, in (\ref{redeq}) in Section \ref{finalred}, we take the infimum as in Definition \ref{sesdef} among these finitely many curves 
 to compute $\varepsilon(L,x)$. Hence $\varepsilon(L,x)$ is actually attained by a specific curve and a rational number.    

When $\sum_{i=1}^{r+1} \frac{b_i}{a} =3$, 
there might be infinitely many classes $(d;m)$ which satisfy (\ref{necineq2}) if we include an equality in the inequality (\ref{necineq2}). 
In order to avoid this, we consider an anticanonical curve. 
For a general point $x \in X$, there exists a curve 
 $C \in |{-}K_X|$ such that $\mult_x(C) =1$ and $C$ satisfies that  
\[
\frac{L \cdot C}{\mult_x C} = \sqrt{L^2}.
\] 
Hence it is enough to consider the classes $(d;m)$ such that 
\begin{equation}\label{necineq3}
d a - \sum_{i=1}^{r+1} m_i b_i < 0  
\end{equation}
and the above anticanonical curve when we take the infimum in (\ref{redeq}) in Section \ref{finalred}. 
In Proposition \ref{bound}, we show that there are only finitely many  
$-1$-classes and $-2$-classes $(d;m_1, \ldots, m_{r+1})$ which satisfy (\ref{necineq3}), 
thus $\varepsilon(L,x)$ turns out to be a rational number.   
\end{rem}

\subsection{Bound of degrees of necessary curve classes}\label{boundsub}
In general, there are infinitely many $-1$-curves and $-2$-curves on $\tilde{X}(x)$. 
However, the following lemma shows that there are only finitely many of those which satisfy the inequality (\ref{necineq}) or (\ref{necineq2}).

\begin{lem}\label{bound} 
Let $X,L$ be as in Setting \ref{setting}. 
\begin{enumerate} 

\item[(i)] Assume that $\sum_{i=1}^{r+1} \frac{b_i}{a} \neq 3$. Then there exists a positive number $M_L$ (resp. $M'_L$) which is determined by $L$ such that,  
 if 
 $(d; m_1,\ldots, m_{r+1})$ is a $-1$-class (resp. $-2$-class) with the inequality $(\ref{necineq2})$, then $d \le M_L$ (resp. $d \le M'_L$). 
\item[(ii)] Assume that $r = 8,9$ and $\sum_{i=1}^{r+1} \frac{b_i}{a} < 3$. 
Set $c:= 3- \sum_{i=1}^{r+1} \frac{b_i}{a}$.  
Then the above $M_L$ can be given by  
\[
M_L := \frac{2- c + \sqrt{c^2 +20c +4}}{2c}.  
\] 
 
\item[(iii)] Assume that $\sum_{i=1}^{r+1} \frac{b_i}{a} = 3$.
Then there exists a positive number $M_L$ (resp. $M'_L$) which is determined by $L$ such that,  
 if 
 $(d; m_1,\ldots, m_{r+1})$ is a $-1$-class (resp. $-2$-class) with the inequality $(\ref{necineq3})$, then $d \le M_L$ (resp. $d \le M'_L$). 
 
 Moreover, $M_L$ can be given by 
 \[
 M_L:= \frac{a^2+1}{2a}. 
 \]
\end{enumerate}
\end{lem}

\begin{proof}
If $r \le 7$, this lemma is trivial since there are only finitely many 
 negative curves on $\widetilde{X}(x)$. 
So we assume $r \ge 8$.

\noindent  
We introduce the notation used in Figure \ref{figcone8}.

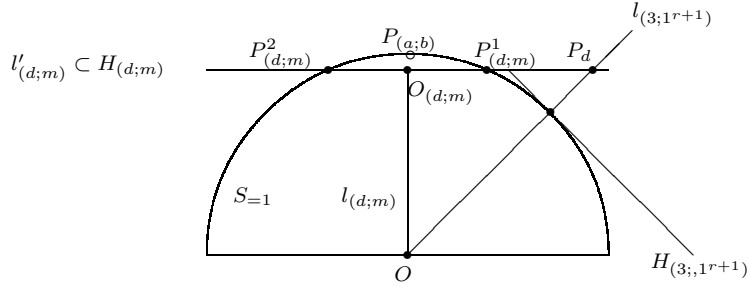
\begin{figure}[t]
\begin{picture}(200,150)(0,70)
   \put(44, 150){\line(1, 0){152}}
\put(120, 80){\line(0, 1){70}}
\put(44, 80){\line(1, 0){152}}
\put(228, 80){\line(-1, 1){70}}
\put(120, 80){\line(1, 1){85}}
    \qbezier(196,80)(196,111.48)(173.74,133.74)
  \qbezier(173.74,133.74)(151.48,156)(120,156)
  \qbezier(120,156)(88.52,156)(66.26,133.74)
  \qbezier(66.26,133.74)(44,111.48)(44,80)
  \put( 212, 75){\tiny $H_{(3;,1^{ r+1})}$}
\put( 205, 170){\tiny $l_{ (3;1^{r+1})}$}
\put( -30, 150){\tiny$l'_{(d;m)} \subset H_{(d;m)}$}
\put(121,155.5){\circle{3}}
\put(120,150){\circle*{3}}
\put(120,80){\circle*{3}}
\put(150,150){\circle*{3}}
\put(190,150){\circle*{3}}
\put( 90, 150){\circle*{3}}
\put(174,134){\circle*{3}}
\put(115,70){\tiny$O$}
\put(120,140){\tiny$O_{(d;m)}$}

\put( 110, 160){\tiny $P_{(a;b)}$}
\put( 180, 155){\tiny $P_{d}$}
\put( 145, 155){\tiny $P^1_{(d;m)}$}
\put( 60, 155){\tiny $P^2_{(d;m)}$}
\put( 50, 100){ \tiny $S_{=1}$}
\put( 95, 100){\tiny $l_{(d;m)}$}
  \end{picture}
\caption{The case $r=8$. Note that $H_{(3;1^{r+1})} \cap S_{=1} = \{(1/3, \ldots, 1/3) \}$  }\label{figcone8}
\end{figure}

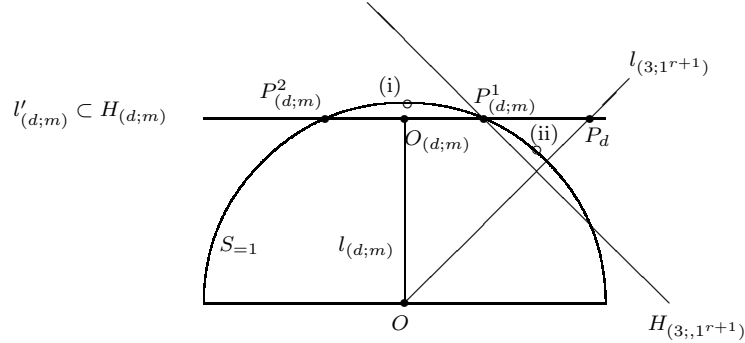
\begin{figure}[t]
\begin{picture}(200,150)(0,70)
   \put(44, 150){\line(1, 0){152}}
\put(120, 80){\line(0, 1){70}}
\put(44, 80){\line(1, 0){152}}
\put(220, 80){\line(-1, 1){114}}
\put(120, 80){\line(1, 1){85}}
    \qbezier(196,80)(196,111.48)(173.74,133.74)
  \qbezier(173.74,133.74)(151.48,156)(120,156)
  \qbezier(120,156)(88.52,156)(66.26,133.74)
  \qbezier(66.26,133.74)(44,111.48)(44,80)
  \put( 212, 70){\tiny$H_{(3;,1^{ r+1})}$}
\put( 205, 170){\tiny $l_{ (3;1^{r+1})}$}
\put( -28, 150.8){\tiny $l'_{(d;m)} \subset H_{(d;m)}$}
\put(121,155.5){\circle{3}}
\put(170,138){\circle{3}}
\put(120,150){\circle*{3}}
\put(120,80){\circle*{3}}
\put(150,150){\circle*{3}}
\put(190,150){\circle*{3}}
\put( 90, 150){\circle*{3}}
\put(115,70){\tiny $O$}
\put( 110, 160){\tiny (i)}
\put(167,142){\tiny (ii)}
\put( 188, 141){\tiny $P_{d}$}
\put( 146, 156){\tiny $P^1_{(d;m)}$}
\put( 65, 158){\tiny $P^2_{(d;m)}$}
\put( 50, 100){\tiny $S_{=1}$}
\put( 95, 100){\tiny $l_{(d;m)}$}
\put(120,140){\tiny $O_{(d;m)}$}
  \end{picture}
\caption{The case $r=9$. White circles (i), (ii) are $P_{(a;b)}$ for the cases $\sum_{i=1}^{r+1} \frac{b_i}{a} < 3$, 
 $\sum_{i=1}^{r+1} \frac{b_i}{a} > 3$, respectively. 
Note that $P^1_{(d;m)} \in H_{(3;1^{r+1})}$  by Claim \ref{lying}.  }\label{figcone9}
\end{figure}

 For  $(d;m)= 
(d;m_1,\ldots,m_{r+1}) \in \mathbb{Z}^{r+2}$, put
\begin{align*}
S_{=1}&:= \{(y_1,\ldots,y_{r+1})\in \mathbb{R}^{r+1} \mid y_1^2+\cdots+y_{r+1}^2=1  \}, \\
N_{(d;m)}&:=\{(y_1,\ldots,y_{r+1})\in \mathbb{R}^{r+1} \mid
m_1 y_1 +\cdots+ m_{r+1} y_{r+1} \ge d \},\\ 
H_{(d;m)}&:=\{(y_1,\ldots,y_{r+1})\in \mathbb{R}^{r+1} \mid
m_1 y_1 +\cdots+ m_{r+1} y_{r+1} = d \}, \\
\Neg_{(d;m)} &:= N_{(d;m)} \cap S_{=1}.   
\end{align*}
For a $-1$-class $(d; m_1, \ldots, m_{r+1})$, we set 
\[
P_{d} := \left( \frac{d}{3d-1},\ldots,\frac{d}{3d-1} \right) = H_{(d;m)} \cap (\mathbb{R}\cdot (1,\ldots,1)).
\] 

Let $l_{(d;m)} \subset \mathbb{R}^{r+1}$ be the line such that 
 $O :=(0,\ldots,0) \in l_{(d;m)}$ 
and $l_{(d;m)} \perp H_{(d;m)}$.  
Set \[
O_{(d;m)}:= \frac{d}{d^2+1}(m_1,\ldots,m_{r+1})= 
 l_{(d;m)} \cap H_{(d;m)}.
 \]
 Let $l'_{(d;m)}$ be the line 
through $P_{d}$ and $O_{(d;m)}$. 

Let $P^1_{(d;m)}, P^2_{(d;m)} \in \mathbb{R}^{r+1}$ be the points such that  
 $l'_{(d;m)} \cap S_{=1} = \{P^1_{(d;m)}, P^2_{(d;m)} \} $. 
 If $r=8,9$, we distinguish these two points by 
 $\delta (P^1_{(d;m)}, H_{(3;1^{r+1})}) < \delta (P^2_{(d;m)}, H_{(3;1^{r+1})})$
 as in Figure \ref{figcone8}, where, for a point $P=(y_1, \ldots, y_{r+1}) \in \mathbb{R}^{r+1}$, we set 
\[
\delta (P, H_{(3;1^{r+1})}):= \left| 3- \sum_{i=1}^{r+1} y_i \right| \ge 0. 
\] 
Note that, for $P \in \mathbb{R}^{r+1}$, we have $P \in H_{(3;1^{r+1})}$ if and only if $\delta (P, H_{(3;1^{r+1})})=0$.

Let $(d; m_1, \ldots, m_{r+1})$ be a $-1$-class with the inequality (\ref{necineq2}). 
Let 
\[
P(a;b):= \left( \frac{b_1}{a},\ldots,\frac{b_{r+1}}{a} \right) \in S_{=1}
\] be the point which corresponds to $L$. 
Since $d a - \sum_{i=1}^{r+1} m_i b_i \le 0$, we have $P(a;b) \in \Neg_{(d; m)}$.

\vspace{5mm}

\noindent(i) 
We first consider the case where $\sum_{i=1}^{r+1} \frac{b_i}{a} \neq 3$.
In this case, we have $P(a;b ) \notin  H_{(3;1^{r+1})}$ and $\delta(P(a;b), H_{(3;1^{r+1}}) >0$. 
Set 
\[
\delta \left( \Neg_{(d; m)}, H_{(3;1^{r+1})}\right):= \sup \left\{ \delta(x, H_{(3;1^{r+1})}) \mid x \in \Neg_{(d;m)}\right\}.  
\]
We show that there are only finitely many $-1$-classes $(d;m)$ such that $P_{(a;b)} \in \Neg_{(d;m)}$ 
by showing that $\delta \left( \Neg_{(d; m)}, H_{(3;1^{r+1})}\right)$ converges to $0$  as $d \rightarrow \infty$.

For a $-1$-class $(d;m)$, we have  
\[
\delta(O_{(d;m)}, H_{(3;1^{r+1})}) = \left| 3- \frac{d}{d^2+1} (3d -1)\right| =  \frac{d+3}{d^2+1}
\]
and this converges to $0$ as $d \rightarrow \infty$. 
The length of the line segment $O O_{(d;m)}$ with respect to the Euclidean metric is 
\[
 \frac{d}{d^2+1} \sqrt{d^2+1} =  \frac{d}{\sqrt{d^2+1}}. 
\]
Hence the lengths of the line segments $O_{(d;m)} P^1_{(d;m)}$ and $O_{(d;m)} P^2_{(d;m)}$ are  
\[
 \frac{1}{\sqrt{d^2+1}} 
\]
and this converges to $0$ as $d \rightarrow \infty$. 
Thus we see that $P^1_{(d;m)}$ and $P^2_{(d;m)}$ converge to the hyperplane $H_{(3;1^{r+1})}$.  
For $i=1,2$, let $H^i_{(d;m)}$ be a hyperplane which is parallel to $H_{(3;1^{r+1})}$ such that $P^i_{(d;m)} \in H^i_{(d;m)}$. 
We can see that $\Neg_{(d;m)}$ lies between $H^1_{(d;m)}$ and $H^2_{(d;m)}$ by elementary calculations. 
Hence we see that  
$\Neg_{(d; m)}$ approaches to the hyperplane $H_{(3;1^{r+1})}$ as $d \rightarrow \infty$ and there are only 
finitely many $d$ such that $P(a;b) \in \Neg_{(d; m)}$. 
Thus we get the claim for $-1$-classes. 

We get the claim for $-2$-classes if $\sum_{i=1}^{r+1} \frac{b_i}{a} \neq 3$, similarly. 

\vspace{5mm}

\noindent(ii) 
Assume that $r=8,9$ and  $\sum_{i=1}^{r+1} \frac{b_i}{a} < 3$. 
For a $-1$-class $(d; m)$ with the inequality (\ref{necineq}), 
we have  
 \[
\delta(P^2_{(d;m)}, H_{(3;1^{r+1})}) \ge \delta (P(a;b),H_{(3;1^{r+1})}) = 3 -  \sum_{i=1}^{r+1} \frac{b_i}{a}.    
\]
We use the following claim. 

\begin{claim}\label{lying}
If $r=9$, we have $P^1_{(d;m)} \in H_{(3;1^{r+1})}$. 
\end{claim} 
\begin{proof}
Set $Q_{(d;m)}:= l'_{(d;m)} \cap H_{(3;1^{r+1})}$.  We show $Q_{(d;m)} = P^1_{(d;m)}$ in the following. 
We consider the length of line segments with respect to an Euclidean metric. 
We see that the length of the line segment $O O_{(d;m)}$ is $d / \sqrt{d^2+1}$ and that of $O_{(d;m)} P^1_{(d;m)}$ is $1/ \sqrt{d^2+1}$ 
as in (i). Hence it is enough to show that the length of the line segment 
$O_{(d;m)} Q_{(d;m)}$ is $1/ \sqrt{d^2+1}$. We can check this by elementary calculations.   
\end{proof}

We can obtain an inequality 
\begin{equation}\label{ineqii}
\delta(P^2_{(d;m)}, H_{(3;1^{r+1})}) \le 2 \cdot \delta(O_{(d;m)}, H_{(3;1^{r+1})})     
\end{equation}
as follows. 
If $r=9$, Claim \ref{lying} actually implies the equality $\delta(P^2_{(d;m)}, H_{(3;1^{r+1})}) = 2 \cdot \delta(O_{(d;m)}, H_{(3;1^{r+1})})$. 
 If $r=8$, we see that $P^1_{(d;m)}$ and $P^2_{(d;m)}$ are on the same side with respect to $H_{(3;1^{r+1})}$ 
  since $\{ (1/3, \ldots ,1/3)\} = H_{(3;1^9)} \cap S_{=1}$. We get the inequality (\ref{ineqii}) by this. 
   
We see that $2 \cdot \delta(O_{(d;m)}, H_{(3;1^{r+1})}) =   2(d+3)/(d^2+1)$ by an easy calculation.  
By (\ref{ineqii}) and this equality,  we have 
\[
c:= 3 -  \sum_{i=1}^{r+1} \frac{b_i}{a} \le \frac{2(d+3)}{d^2+1}.  
\]
 Hence we can take $M_L$ as the solution of the equation 
\[
c = \frac{2(x+3)}{x^2+1}, x>0.  
\]

\vspace{5mm}

\noindent(iii)
We next consider the case $\sum_{i=1}^{r+1} \frac{b_i}{a} = 3$. 

Let $(d;m_1, \ldots, m_{r+1})$ be a $-1$-class with the inequality (\ref{necineq3}). 
We see that  $P(a;b) \notin H_{(d;m)}$ and 
 $P(a;b)$ lies between $H_{(d;m)}$ and $H_{(\sqrt{d^2+1};m)}$ since we see that  
\[
H_{(\sqrt{d^2+1};m)} \cap S_{=1} = \left\{ \frac{1}{\sqrt{d^2+1}} (m_1,\ldots,m_{r+1}) \right\} 
\] 
by an elementary computation. 
Thus we have 
\[
0 < \delta (P(a;b),H_{(d;m)}) \le 
\delta (H_{(d;m)},H_{(\sqrt{d^2+1};m)}) = 
 \sqrt{d^2+1} - d =  \frac{1}{\sqrt{d^2+1}+d}.
 \] 
Moreover we have  
\[
\delta(P(a;b),H_{(d;m)}) = 
\frac{| \sum_{i=1}^{r+1} m_i b_i - d a |}{a} \ge \frac{1}{a} 
\] 
since $ \sum_{i=1}^{r+1} m_i b_i - d a  \in \mathbb{Z}$ for all $i$. 
Indeed, $b_{r+1} = 3 a - \sum_{i=1}^r b_i$ is an integer, and so $a, d, b_i, m_i \in \mathbb{Z}$ for $i =1,\ldots, r+1$. 
Therefore $\sqrt{d^2+1}+d \le a$. 
Hence we can take $M_L$ as the solution of the equation 
\[
\sqrt{x^2+1}+x = a. 
\] 

We can get the statement for $-2$-classes similarly. 
\end{proof}

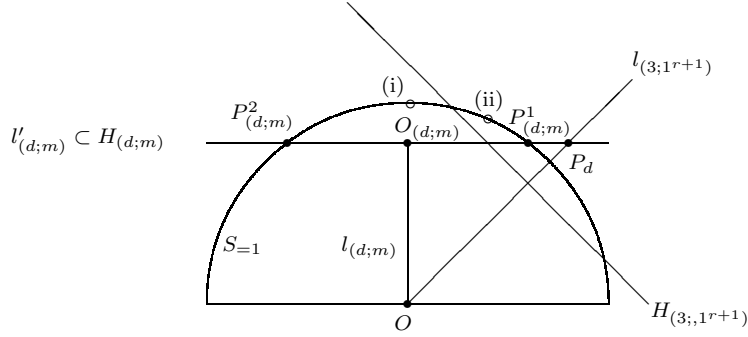
\begin{figure}[t]
\begin{picture}(200,150)(0,70)
   \put(44, 140.8){\line(1, 0){152}}
\put(120, 80){\line(0, 1){60.8}}
\put(44, 80){\line(1, 0){152}}
\put(211.2, 80){\line(-1, 1){114}}
\put(120, 80){\line(1, 1){85}}
    \qbezier(196,80)(196,111.48)(173.74,133.74)
  \qbezier(173.74,133.74)(151.48,156)(120,156)
  \qbezier(120,156)(88.52,156)(66.26,133.74)
  \qbezier(66.26,133.74)(44,111.48)(44,80)
  \put( 212, 75){\tiny$H_{(3;,1^{ r+1})}$}
\put( 205, 170){\tiny$l_{ (3;1^{r+1})}$}
\put( -30, 140.8){\tiny $l'_{(d;m)} \subset H_{(d;m)}$}
\put(121,155.5){\circle{3}}
\put(150.4,150){\circle{3}}
\put(120,140.8){\circle*{3}}
\put(120,80){\circle*{3}}
\put(165.6,140.8){\circle*{3}}
\put(180.8,140.8){\circle*{3}}
\put( 74.4, 140.8){\circle*{3}}
\put(115,70){\tiny $O$}
\put(115,145){\tiny $O_{(d;m)}$}
\put( 110, 160){\tiny (i)}
\put(145,155){\tiny (ii)}
\put( 180.8, 131){\tiny $P_{d}$}
\put( 158, 148){\tiny $P^1_{(d;m)}$}
\put( 53, 150){\tiny $P^2_{(d;m)}$}
\put( 50, 100){\tiny $S_{=1}$}
\put( 95, 100){\tiny $l_{(d;m)}$}
  \end{picture}
\caption{The case $r \ge 10$.  
White circles (i), (ii) are $P_{(a;b)}$ for the cases $\sum_{i=1}^{r+1} \frac{b_i}{a} < 3$, 
$\sum_{i=1}^{r+1} \frac{b_i}{a} > 3$, respectively. }\label{figcone10}
\end{figure}

\begin{rem}\label{excrem}
Lemma \ref{bound} does not give an explicit description of $M_L$  for the cases where $r \ge 9, \sum_{i=1}^{r+1} \frac{b_i}{a} >3$. 

If $r=9$ and  $\sum_{i=1}^{10} \frac{b_i}{a} >3$, then a $-1$-class $(d; m_1, \ldots, m_{10})$ satisfies that 
\[
d a - \sum_{i=1}^{10} m_i b_i >0   
\]
since we have $\Neg_{(d;m)} \subset N_{(3;1^{10})}$. Indeed this follows from Claim \ref{lying}. 
Hence we do not need to consider $-1$-classes for the computation of $\varepsilon_{\gen}(L)$ for such $L$. 

If $r \ge 10$ and $ \sum_{i=1}^{r+1} \frac{b_i}{a} >3$, the above argument in the proof of (ii) does not work since $P^1_{(d;m)}$ and 
$P^2_{(d;m)}$ are on different sides with respect to $H_{(3;1^{r+1})}$ as shown in Figure \ref{figcone10}. 
We can give an explicit description of $M_L$ by more detailed calculation. 
However we omit it for simplicity.  
\end{rem}

\subsection{Final reduction}\label{finalred}

By Section \ref{reduction1} (\ref{reduce1}) and Lemma \ref{bound}, we see that the computation of $\varepsilon(L,x)$ is reduced to consider 
an anticanonical curve on $X$ and a curve $\tilde{C} \subset \tilde{X}(x)$ which is either a $-1$-curve, a $-2$-curve of low degree or a fixed component 
of $|{-}K_{\tilde{X}(x)}|$. 

We see that a $-1$-class  is always effective by the following lemma. 

\begin{lem}
Let $(d;m_1, \ldots, m_s) \in \mathbb{Z}^{s+1}$ be a $-1$-class on a smooth rational surface $X$ with a birational morphism 
$\mu \colon X \rightarrow \mathbb{P}^2$ such that $K_X^2=9-s$. Let $E_1, \ldots, E_s$ be exceptional divisors as in Section \ref{state}. Set 
\[
\mathcal{O}_X(d;m_1, \ldots, m_s) := \mu^* \mathcal{O}_{\mathbb{P}^2}(d) - \sum_{i=1}^s m_i E_i. 
\] 
Then we see that 
\[
h^0(X, \mathcal{O}_X(d;m_1, \ldots, m_s)) >0. 
\] 
\end{lem}

\begin{proof}
It follows from 
\[
h^0(X, \mathcal{O}_X(d;m_1, \ldots, m_s)) \ge \frac{(d+1)(d+2)}{2} - \sum_{i=1}^s \frac{m_i(m_i+1)}{2} =1 
\]
since we have $d^2 - \sum_{i=1}^s m_i^2 = -1, 3d- \sum_{i=1}^s m_i =1$. 
\end{proof}

Hence we see that 
\begin{equation}\label{redeq}
\varepsilon(L,x) = \min \left\{ 
A(L), -K_X \cdot L, \inf \frac{L \cdot C}{\mult_x C}
\right\}, 
\end{equation}    
where  
\[
A(L) := \inf_{(d;m) \in \Phi_{r+1}} 
\frac{d a - \sum_{i=1}^r m_i b_i}{m_{r+1}} 
\]
is the contribution of $-1$-cycles on $\tilde{X}(x)$ and the infimum in the third term is taken over all curves $C$ in $X$ through $x$ 
whose strict transform $\tilde{C}$ in $\tilde{X}(x)$ 
is   
either a $-2$-curve of low degree or a fixed component of $|{-}K_{\tilde{X}(x)}|$. 
Note that $\varepsilon(L,x)$ does not change whether we consider $-1$-cycles or $-1$-curves since 
the constant $\varepsilon(L,x)$ does not change if we consider reducible curves $C$ on $X$ as explained in Remark \ref{reducible}.

In the following, we show that we do not need to consider the above third term in a general point $x$ in the equality (\ref{redeq}). 

\begin{prop}\label{2tsu}
Let $X,L$ be as in Setting \ref{setting}.  Then there exists a nonempty Zariski open subset 
$U_L \subset X$ such that, for $x \in U_L$,  
\[
\varepsilon_{\gen}(L) = \varepsilon(L,x) = \min \{ A(L), -K_X \cdot L \} .
\] 
\end{prop}

\begin{proof}
Put   
\[
U: = \{ x \in X \mid \exists D \in |{-}K_X| \text{ such that } \mult_x(D) =1 \} \setminus \Bs |{-}K_X|,  
\]
where $\Bs |{-}K_X|$ is the base locus. 
We see that $U$ is a Zariski open subset. 

Let $x \in U$ be a point such that 
\begin{equation}\label{excineq}
\varepsilon(L,x) < \min \{A(L), -K_X \cdot L \}.
\end{equation}
 Then there exists a curve $C$ through $x$ such that 
its strict transform $\tilde{C} \subset \tilde{X}(x)$ satisfies 
\begin{equation}\label{sesexc}
(\mu_x^* L - \sqrt{L^2} E_x) \cdot \tilde{C} \le 0.
\end{equation} 
Such $\tilde{C}$ with the inequality (\ref{sesexc}) is either a $-2$-curve or a fixed component of $|{-}K_{\tilde{X}(x)}|$ 
by the description (\ref{redeq}) of $\varepsilon(L,x)$. 
Note that we have a strict inequality in (\ref{sesexc}) when $\sum_{i=1}^{r+1} \frac{b_i}{a}=3$ as explained in Remark \ref{=3rem}. 

\vspace{5mm}

First consider when such a $\tilde{C}$ in (\ref{sesexc}) is a $-2$-curve. We see that $d \le M'_{L}$ by Lemma \ref{bound} (i) and (iii),   
where $\tilde{C} \in |\mathcal{O}_{\tilde{X}(x)}(d;m)|$ and $M'_L$ is the number given in Lemma \ref{bound}. 
Hence there are only finitely many possibilities for classes $(d;m_1,\ldots,m_{r+1})$ of $\tilde{C}$. 
Moreover, by the following claim, we see that such $\tilde{C}$ does not appear on $\tilde{X}(x)$ for a general point $x \in X$. 

\begin{claim}\label{-2vanish}
Let $(d; m_1, \ldots ,  m_{r+1})\in \mathbb{Z}^{r+2}$ be a $-2$-class such that $m_{r+1}\ge 1$. 
Put 
\[
Z_{(d;m)}:= \{x \in U \mid \exists  {-}2 \text{-curve } \tilde{C} \in |\mathcal{O}_{\tilde{X}(x)}(d;m_1,\ldots, m_{r+1})| \} . 
\]
Then $\overline{Z_{(d;m)}}$ is a proper Zariski closed subset of $X$. 
\end{claim}

\begin{proof}[Proof of Claim]
Let 
\[
-K_X = M +F 
\]
be the decomposition such that $M$ is the moving part of $|{-}K_X|$ and $F$ is the fixed part of $|{-}K_X|$. 
Then we have $h^0(X,M) = h^0(X, -K_X) \ge 2$ and general elements $\Theta_M \in |M|$ are reduced and smooth outside $\Bs |{-}K_X|$. 
For such $\Theta_M$, we consider $Z_{(d; m)} \cap \Theta_M^{\sm}$, where $\Theta_M^{\sm}$ is the smooth locus of $\Theta_M$. 
It is enough to show that $Z_{(d; m)} \cap \Theta_M^{\sm}$ 
is a finite set. 

Take $x \in Z_{(d;m)} \cap \Theta_M^{\sm}$. 
Set $\Theta:= \Theta_M + F \in |{-}K_X|$. Let $\tilde{\Theta} \subset \tilde{X}(x)$ be the strict transform of $\Theta$. 
Since $x \in \Theta_M^{\sm}$, we see that $\tilde{\Theta} \in |{-}K_{\tilde{X}(x)}|$. 
Let $\tilde{C}$ be a $-2$-curve of type $(d;m)$ which exists since $x \in Z_{(d;m)}$. 
 we can show that 
\begin{equation}\label{isomtheta}
\mathcal{O}_{\tilde{X}(x)}(\tilde{C})|_{\tilde{\Theta}} \simeq \mathcal{O}_{\tilde{\Theta}}   
\end{equation}
as follows. Note that $\tilde{C}$ is not contained in $\tilde{\Theta}$ as its irreducible component since $x \in \Theta_M$. 
Hence the restriction $\mathcal{O}_{\tilde{X}(x)}(\tilde{C})|_{\tilde{\Theta}_i} $ is effective, where $\tilde{\Theta}_i$ is the irreducible component of $\tilde{\Theta}$.   
By this and $-K_{\tilde{X}(x)} \cdot \tilde{C} =0$, each sheaf $\mathcal{O}_{\tilde{X}(x)}(\tilde{C})|_{\tilde{\Theta}_i}$ has degree $0$. Hence we get an isomorphism (\ref{isomtheta}).  

By the isomorphism (\ref{isomtheta}) and an isomorphism 
$\tilde{\Theta} \simeq \Theta$ as schemes,  
 we obtain   
\[
\mathcal{O}_X(d; m_1, \ldots, m_r)|_{\Theta} \simeq \mathcal{O}_{\Theta}(m_{r+1} \cdot x) 
\]
Thus, for $x' \in Z_{(d;m)} \cap \Theta_M^{\sm}$, we see that $\mathcal{O}_{\Theta}(x-x')$ is 
a torsion point of $\Pic^0 \Theta$ of order $m_{r+1}$. 
Note that $\dim \Pic^0 \Theta = h^1(\Theta, \mathcal{O}_{\Theta}) =1$ and there are only finitely many torsion points since the characteristic is zero. 
Let $\Theta^x_M$ be the irreducible component of $\Theta_M$ which contains $x$. 
We get the claim since the morphism 
\[
(\Theta^x_M)^{\sm} \rightarrow \Pic^0 \Theta ; p \mapsto \mathcal{O}_{\Theta}(x-p) 
\]
is non-constant. Indeed, for $p \in (\Theta_M^x)^{\sm}$, we see that $h^0(\Theta, \mathcal{O}_{\Theta}(p)) =1$ since we have an exact sequence 
\begin{equation}
0 \rightarrow H^0(\mathcal{O}_{\Theta}) \rightarrow H^0(\mathcal{O}_{\Theta}(p)) \rightarrow H^0( \mathcal{O}_{\Theta,p} / \mathfrak{m}_{\Theta,p}) 
\rightarrow H^1(\mathcal{O}_{\Theta}) \rightarrow H^1(\mathcal{O}_{\Theta}(p)) =0.  
\end{equation}
Hence we finish the proof of Claim \ref{-2vanish}. 
\end{proof}

Next we consider the case where $\tilde{C}$ with (\ref{sesexc}) satisfies  $\tilde{C} \subset \Fix |{-}K_{\tilde{X}(x)}|$. 
We see that $\tilde{C} \notin |{-}K_{\tilde{X}(x)}|$ by (\ref{excineq}), 
hence $C$ is rational. Moreover, $C$ is smooth since $x \in U$. 
If $\tilde{C}^2 = C^2 -1 \le -4$, then $C \subset \Fix |{-}K_X|$ since 
$-K_X \cdot C <0$. This contradicts that $ x \in U$. 
Hence $\tilde{C}$ is a $-3$-curve and $C$ is a $-2$-curve satisfying $L \cdot C \le \sqrt{L^2}$. There are only finitely many such curves and we set them as 
$C_1, \ldots, C_{d(L)}$ 
for some integer $d(L)$ determined by $L$.    

\vspace{5mm}
Set 
\[
U_L:= U \setminus \left( \bigcup_{d \le N_{L}} \overline{Z_{(d;m)}} \cup 
\bigcup_{i=1}^{d(L)} C_i \right). 
\]
Then  $U_L$ satisfies the required condition. Indeed there does not exist an irreducible curve $C \subset X$ through $x \in U_L$ whose 
strict transform on $\tilde{X}(x)$ 
is either a $-2$-curve or a $-3$-curve that  
satisfies the 
inequality (\ref{excineq}). 
\end{proof}

\begin{proof}[Proof of Theorem \ref{main}]
By Lemma \ref{bound}, Remark \ref{excrem} and   Proposition \ref{2tsu}, we get the result.
\end{proof}

\begin{rem}
We can compute $\varepsilon_{\gen}(L)$ for an ample $\mathbb{Q}$-divisor $L$ by computing $\varepsilon_{\gen}(mL)$ for a positive integer $m$ 
which is minimal among those such that $mL$ is integral.

Note that we used the assumption that $L$ is integral and primitive only in the case $\sum_{i=1}^{r+1} \frac{b_i}{a} =3$ in Lemma \ref{bound}. 
Hence,    
for an ample $\mathbb{R}$-divisor $L$ such that $-K_{\tilde{X}(x)} \cdot (\mu_x^*L - \sqrt{L^2} E_x) \neq 0$, we can compute 
$\varepsilon_{\gen}(L)$.  
\end{rem}

\subsection{The case of a blow-up of $\mathbb{F}_n$}\label{hirzsection}
We can also compute $\varepsilon_{\gen}(L)$ on $X$ such that $h^0(X,-K_X) \ge 2$ with a birational morphism $X \rightarrow \mathbb{F}_n$, where 
$\mathbb{F}_n:= \mathbb{P}(\mathcal{O}_{\mathbb{P}^1} \oplus \mathcal{O}_{\mathbb{P}^1}(-n))$ is the Hirzebruch surface. 
Let $\pi \colon  \mathbb{F}_n \rightarrow \mathbb{P}^1$ be the projection of the $\mathbb{P}^1$-bundle structure. 

\begin{set}\label{setting2}  
Let $X$ be a smooth rational surface such that $\dim |{-} K_X | \ge 1$ with a birational morphism $\nu \colon X \rightarrow \mathbb{F}_n$ which is a composition of blow-ups  
 \begin{equation}
 \nu \colon X=Y_{r+1} \stackrel{\nu_r}{\rightarrow} Y_r \rightarrow 
\cdots \stackrel{\nu_1}{\rightarrow} Y_1 
  = \mathbb{F}_n, 
  \end{equation} 
  where $\nu_i:Y_{i+1} \rightarrow Y_i$
is the blow-up at $y_i \in Y_i$. 
Let
 \begin{equation}\label{ampleform2}
 L:= \nu^{\ast} \mathcal{O}_{\mathbb{F}_n}(a H +b F) - \sum_{i=1}^r c_i G_i 
 \end{equation}
 be an ample Cartier divisor on $X$, where $a,b, c_1,\ldots, c_r$ are integers, $H:= \mathcal{O}_{\mathbb{P}(\mathcal{O}_{\mathbb{P}^1} \oplus \mathcal{O}_{\mathbb{P}^1}(-n))}(1)$ is the tautological bundle, $F:= \pi^* \mathcal{O}_{\mathbb{P}^1}(1)$ is the fiber class and   
 \[
G_i:= (\nu_{i+1} \circ \cdots \circ \nu_r)^* \nu_i^{-1}(y_i)
\]
is the exceptional divisor for $i=1,\ldots,r$ which is a $-1$-cycle. 
Let $\mu_x \colon \tilde{X}(x) \rightarrow X$ be the blow-up at a point $x \in X$ with the exceptional curve $E_x := \mu_x^{-1}(x)$. 
For integers $k_1,k_2, l_1, \ldots, l_{r+1}$,  we set 
\[
\mathcal{O}_{\tilde{X}(x)}(k_1,k_2;l_1, \ldots, l_{r+1}):= \mu_x^* \left(\nu^* (k_1 H + k_2 F) - \sum_{i=1}^{r} l_i G_i \right) -l_{r+1} E_x. 
\] 
\end{set}

The consideration in (\ref{reduce1}) about $\varepsilon(L,x)$ also works in this case. 
We see that $|{-}K_{\tilde{X}(x)}|$ does not have a fixed component for a general $x \in X$ by the proof of Lemma \ref{2tsu}. 
Hence, in order to compute $\varepsilon_{\gen}(L)$, it is enough to consider $-1$-curves and $-2$-curves on $\tilde{X}(x)$ and an element of $|{-}K_X|$.  

\vspace{5mm}

\noindent {\bf Case: $r \le n-2$.}
If $r \le n-2$, we have the following. 

\begin{prop}\label{n-2prop}
Let $X$ and $L$ be as in Setting \ref{setting2}. Assume that $r \le n-2$. 
Let $\mu_x \colon \tilde{X}(x) \rightarrow X$ be the blow-up at a point $x \in X$ with the exceptional curve $E_x := \mu_x^{-1}(x)$.  
\begin{enumerate}
\item[(i)] If $x$ is general, there is no $-2$-curve $\tilde{C}$ on $\tilde{X}(x)$ such that $\tilde{C} \cdot E_x \ge 1$.  
\item[(ii)] Let $\tilde{C}$ be a $-1$-curve on $\tilde{X}(x)$ such that $\tilde{C} \cdot E_x \ge 1$. 
Then we have 
\[
\tilde{C} \in \left| \mathcal{O}_{\tilde{X}(x)}(0,1;0^{\times r},1) \right|.
\]
\end{enumerate} 
\end{prop}

\begin{proof}
\noindent(i) Suppose that there exists such a $-2$-curve $\tilde{C}$ on $\tilde{X}(x)$. 
Set 
\[
\mathcal{O}_{\tilde{X}(x)}(\tilde{C}) := \mathcal{O}_{\tilde{X}(x)}(\alpha, \beta; \gamma_1, \ldots, \gamma_{r+1}). 
\] 
Since $x$ is general, we can easily see that $\alpha \neq 0$ and $\beta \neq 0$. 
Thus $\nu(\mu_x(\tilde{C})) \subset \mathbb{F}_n$ is an irreducible reduced curve. 
It is linearly equivalent to neither $H$ nor $F$. Hence we have 
\begin{equation}\label{HFineq}
-n\alpha+\beta = (\alpha H + \beta F)\cdot H \ge 0,
 \alpha =(\alpha H + \beta F)\cdot F \ge 0. 
\end{equation} 
Since $\tilde{C}$ is a $-2$-curve, we have 
\begin{equation}\label{acrel0}
0 = -K_{\tilde{X}(x)} \cdot \tilde{C} = 2 \beta - (n-2) \alpha - \sum_{i=1}^{r+1} \gamma_i. 
\end{equation}
Note that $-1$-cycles $\mu_x^*( \nu^* F - G_i)$ for $i=1, \ldots, r$ and $\mu_x^* \nu^* F - E_x$ are effective and 
they do not have $\tilde{C}$ as their irreducible component. Hence we have $\alpha \ge \gamma_i$ since we have 
\[
\alpha- \gamma_i = \mu_x^*( \nu^* F - G_i) \cdot \tilde{C} \ge 0 \ \ \ \ i=1,\cdots, r, 
\] 
\[
\alpha - \gamma_{r+1} = (\mu_x^* \nu^* F - E_x) \cdot \tilde{C} \ge 0.
\]  

Since we have $\beta \ge n\alpha$ and $\alpha >0$ by the above arguments, we see that  
\begin{equation}\label{finalineq1}
2 \beta - (n-2) \alpha - \sum_{i=1}^{r+1} \gamma_i \ge (n+2) \alpha - \sum_{i=1}^{r+1} \gamma_i \ge (n-r+1) \alpha \ge 3 \alpha >0.
\end{equation}
This contradicts (\ref{acrel0}) and we get the claim (i). 

\noindent(ii) Suppose that there exists a $-1$-curve $\tilde{C}$ such that $\tilde{C} \cdot E_x \ge 1$ and 
$\mathcal{O}_{\tilde{X}(x)}(\tilde{C}) \not \simeq \mathcal{O}_{\tilde{X}(x)}(0,1;0^{r},1)$. 
We again set $\mathcal{O}_{\tilde{X}(x)}(\tilde{C}) := \mathcal{O}_{\tilde{X}(x)}(\alpha, \beta; \gamma_1, \ldots, \gamma_{r+1})$. 

 As in (i), we get $\beta \ge n\alpha, \alpha \ge 0, \alpha \ge \gamma_i$ and  
\begin{equation}\label{acrel1}
1 = -K_{\tilde{X}(x)} \cdot \tilde{C} = 2 \beta - (n-2) \alpha - \sum_{i=1}^{r+1} \gamma_i. 
\end{equation}
As in (i), we have the same inequality (\ref{finalineq1}) and $3 \alpha >1$. It contradicts (\ref{acrel1}).  
Thus we get the claim (ii). 
\end{proof}

By the above argument and Proposition \ref{n-2prop}, we get the following. 
\begin{cor} Let $X$ and $L$ be those as in Setting \ref{setting2}. If $r \le n-2$, we have  
\[
\varepsilon_{\gen}(L) = \frac{\nu^* F \cdot L}{\mult_x \nu^*F} = a.
\] 
\end{cor}
 
\vspace{5mm} 
 
\noindent{\bf Case: $r \ge n-1$.}
We sketch how to compute $\varepsilon_{\gen}(L)$ when $r \ge n-1$.  
We can show that there are only finitely many $-2$-classes $\tilde{C}:=\mathcal{O}_{\tilde{X}(x)}(\alpha, \beta; \gamma_1, \ldots, \gamma_{r+1})$
 such that $(\mu_x^* L - \sqrt{L^2} E_x)\cdot \tilde{C} \le 0$ by showing the analogous statement as Lemma \ref{bound}. 
However we omit the details. 

Hence we have \[
\varepsilon_{\gen}(L)= \min \{ A(L), -K_X \cdot L \},  
\]
where $A(L):= \inf L \cdot C / \mult_x (C)$ is the infimum taken over all curves $C$ through $x$ 
whose strict transforms on $\tilde{X}(x)$ are $-1$-cycles. 
Since $-1$-cycles are effective and the contribution of $-1$-cycles are determined by the numerical structure of the lattice $\Pic \tilde{X}(x) \simeq \mathbb{Z}^{r+3}$, 
we can compute $A(L)$ by deforming the locations of centers $y_1, \ldots, y_r$. 

Thus, we can replace $X$ with a small deformation of $X$ with a birational morphism to $\mathbb{P}^2$  
 by deforming $y_1, \ldots, y_r$ so that the images of $y_1, \ldots, y_r$ on $\mathbb{F}_n$ are on distinct fibers of $\pi$ and 
not on the negative section of $\mathbb{F}_n$. 
Indeed, we get a birational morphism to $\mathbb{P}^2$ by performing elementary transformations on $y_1, \ldots, y_{n-1}$. 
Denote the obtained morphism as 
\[
\mu \colon X= X_{r+2} \stackrel{\mu_{r+1}}{\rightarrow} X_{r+1}\rightarrow \cdots \stackrel{\mu_1}{\rightarrow} X_1 = \mathbb{P}^2, 
\] 
where $\mu_i \colon X_{i+1} \rightarrow X_i$ is the blow-up at $x_i \in X_i$.  
By the property of elementary transformations, we can assume that the induced birational morphism 
$X_n \rightarrow X_2$ is a composition of blow-ups of distinct $n-1$ points on a $-1$-curve $\mu_1^{-1}(x_1) \subset X_2$. 
Let $E_i:=(\mu_{i+1} \circ \cdots \circ \mu_{r+1})^* \mu_i^{-1}(x_i)$ be a $-1$-cycle for $i=1, \ldots, r+1$ as in Setting \ref{setting}.  

We have the following equalities between two bases of $\Pic X$ which are coming from $\nu \colon X \rightarrow \mathbb{F}_n$ and 
$\mu \colon X \rightarrow \mathbb{P}^2$; 
\[
\nu^* H = E_1 - \sum_{i=2}^{n} E_i, \nu^* F = \mu^* \mathcal{O}_{\mathbb{P}^2}(1) - E_1, 
\]
\[
G_1= \mu^* \mathcal{O}_{\mathbb{P}^2}(1)-E_1 -E_2, \ldots, G_{n-1} = \mu^* \mathcal{O}_{\mathbb{P}^2}(1)- E_1 -E_n, 
\]
\[
G_n = E_{n+1}, \ldots, G_r= E_{r+1}. 
\]
By this description, we can rewrite $L$ as a linear combination of the basis $\mu^*\mathcal{O}_{\mathbb{P}^2}(1), E_1, \ldots, E_{r+1}$. 
Hence we can compute $A(L)=A_{M_L}(L)$ and $\varepsilon_{\gen}(L)$ by using Theorem \ref{main}.

\subsection{Examples}\label{exsection}
We compute $\varepsilon_{\gen}(L)$ for a specific $X$ and $L$ in the following examples. 

\begin{eg}\label{compex5}
We consider the case $r=5$ which we need in the proof of Theorem \ref{ldpclassif}. 
Let $L:= \mu^* \mathcal{O}_{\mathbb{P}^2}(5) - \sum_{i=1}^5 E_i$. Assume that $L$ is nef. 
In order to compute $\varepsilon_{\gen}(L)$,  we should consider the anticanonical class and $-1$-classes of the form $(1;1,0^4,1)$, $(2;1^4,0,1)$ or its permutations. 
Hence we see that the $-1$-class $(1;1,0^4,1)$ computes $\varepsilon_{\gen}(L)$ and  
\[
\varepsilon_{\gen}(L) = 5-1 =4.  
\] 
\end{eg}

\begin{eg}\label{compex}
Let $X = S_1$ be the smooth del Pezzo surface of degree $1$. Then $\mu: X \rightarrow \mathbb{P}^2$ 
is the blow-up at general $8$ points $x_1, \ldots, x_8$. 
Now put $L = \mu^{\ast}\mathcal{O}_{\mathbb{P}^2}(4) - \sum_{i=1}^8 E_i$. Then $L$ is an ample divisor on $X$. 
We compute $\varepsilon_{\gen}(L)$. First, we calculate $M_{L}$. 
In this case, we can compute 
\[
\lfloor M_{L} \rfloor 
= 8,   
\] 
where $\lfloor M_L \rfloor$ is the integer such that $\lfloor M_L \rfloor \le M_L < \lfloor M_L \rfloor +1$ 
and we can list the members of $\Phi_9(8) = \{(d;m) \in \Phi_9; d \le 8 \} $. 
By computation, we see that  the element $(4;1^8,3)$ computes $\varepsilon_{\gen}(L)$ as 
\[
\varepsilon_{\gen}(L) = \frac{4 \times 4 - 1 \times 8 }{3} = \frac {8}{3}.
\]
\end{eg}

\begin{rem}(\cite{Br})\label{Br}
If $\mu:X=X_{r+1} \rightarrow \mathbb{P}^2$ is the blow-up at general $r$ points ($r \le 8$), i.e.\ 
$X$ is a del Pezzo surface, then 
\[
\varepsilon_{\gen}(-K_X) = 
\begin{cases}
2 &(1 \le r \le 5)\\
\frac{3}{2} &(r=6)\\
\frac{4}{3} &(r=7)\\
1 &(r=8)
\end{cases}
\]
We put $a_r := \varepsilon_{\gen}(-K_{X_{r+1}}) (r=1,\ldots,8)$. 
\end{rem}

Then the following is a consequence of Theorem \ref{main} since we can compute $\varepsilon_{\gen}(L)$ 
from the numerical data. 

\begin{cor}\label{same}
Let $X$ be a smooth weak del Pezzo surface, that is, $-K_X$ is nef and big, and put $r:= 9- K_X^2$. Then 
\[
\varepsilon_{\gen}(-K_X) = a_r, 
\]
where $a_r$ is the number defined in Remark \ref{Br}.
\end{cor}

\section{Characterization of log del Pezzo surfaces with mild singularities}\label{ldpstate}

Seshadri constants satisfy the following lower semicontinuity property with respect to deformations. 

\begin{prop}\label{lsc}$($\cite{Laz}, Example $5.1.11 )$ 
Let $f:\mathcal{X} \rightarrow T$ be a flat projective morphism of algebraic varieties,  
$s: T \rightarrow \mathcal{X}$ be a section of $f$, that is $f \circ s = id_T$,  
and $\mathcal{L}$ be a $f$-ample invertible sheaf on $\mathcal{X}$. 
Assume $f$ is smooth along $s(T)$. Then, 
\[
\varepsilon_{\mathcal{L},s} : T_{\text{ct}} \rightarrow \mathbb{R}_{> 0}; 
t \mapsto \varepsilon(\mathcal{L}|_{f^{-1}(t)}, s(t))
\] 
is a lower semicontinuous function. Here $T_{\text{ct}}=T$ as a set and the topology of $T_{\text{ct}}$ is 
determined by the rule that 
the closed subsets of $T_{\text{ct}}$ 
are countable unions of Zariski closed subsets of $T$.   
\end{prop}

\begin{defn}
Let $X$ be a normal projective surface. 
$X$ is called a {\it log del Pezzo surface} if its anticanonical divisor is ample and $X$ has only quotient singularities.  
\end{defn}

For explaining the motivation, let $X$ be a log del Pezzo surface with a $\mathbb{Q}$-Gorenstein smoothing, that is, 
there is a flat projective morphism $f:\mathcal{X} \rightarrow T \ni 0$ such that the relative canonical divisor 
$-K_{\mathcal{X}/T}$ is $\mathbb{Q}$-Cartier, $f$-ample, $\mathcal{X}_0 \simeq X$, 
and $\mathcal{X}_t$ are smooth for all $t \neq 0$. Hacking and Prokhorov classified such log del Pezzo surfaces of Picard number $1$ 
over $\mathbb{C}$ in \cite{HP}. They are important objects in the classification of $3$-fold del Pezzo fibrations.   
By Proposition \ref{lsc}, we have 
\[
\varepsilon_{\text{gen}}(-K_{\mathcal{X}_0}) \le \varepsilon_{\text{gen}}(-K_{\mathcal{X}_t}). 
\]     
The following question is the main subject of this section. 
\begin{prob}\label{equal}
Determine $X$ such that equality holds in the above inequality.
\end{prob}

\begin{rem}
It is known that a log del Pezzo surface $X$ over $\mathbb{C}$ has a $\mathbb{Q}$-Gorenstein smoothing if and only if 
 $X$ has only T-singularities, i.e.\ $X$ has only Du Val singularities or cyclic quotient singularities 
 of type $\frac{1}{dn^2}(1,dna-1)$ for some $d,n,a \in \mathbb{Z}_{>0}$ such that $(a,n) =1$. 
 For details, see \cite{HP}. 
\end{rem}

We answer Problem \ref{equal} when $4 \le K_X^2 \le 9$. In these cases, we do not need the assumption 
that $X$ has only T-singularities.

\begin{thm}\label{ldpclassif}
\begin{enumerate}

\item[(I)] Let $X$ be a log del Pezzo surface such that $K_X^2 =4,5,6,7,8$. Then, the following hold. 
\begin{enumerate}
\item[(a)] Suppose that $X$ has only canonical singularities or $ X $ is one of the special 7 types of the surfaces 
$Z_i (i=1,\ldots,7)$ which are defined in Remark \ref{speldp}.   
Then $\varepsilon_{\gen}(-K_X) = 2.$  
\item[(b)] Otherwise, we have $\varepsilon_{\gen}(-K_X) < 2.$
\end{enumerate}
\item[(II)] Let $X$ be a log del Pezzo surface such that $K_X^2 = 9$. Then, the following hold. 
\begin{enumerate}
\item[(a)] If $X \simeq \mathbb{P}^2$, then  $\varepsilon_{\gen}(-K_X) = 3.$
\item[(b)] Otherwise, we have $\varepsilon_{\gen}(-K_X) \le 2.$ 
\end{enumerate}
\end{enumerate}
\end{thm}

\begin{rem}\label{speldp}
We describe the special log del Pezzo surfaces $Z(k_1,\ldots,k_m)$,   
where $(k_1,\ldots,k_m)$ is one of 
\[
(1,1,1,1,1), (2,1,1,1), (3,1,1),(2,2,1),(4,1), (3,2), \text{or } (5).
\] 
In Figure \ref{ldpzu}, we give a picture of $Z(2,1,1,1)$. 

\begin{figure}[t]
\begin{picture}(380,240)(0,0)
\put(130, 180){\vector(1, 0){10}}
\put(250, 180){\vector(1, 0){10}}
\put(80, 130){\vector(0, -1){10}}

\put(280, 140){\line(0, 1){80}}
\put(360, 220){\line(0, -1){80}}
\put(280, 140){\line(1, 0){80}}
\put(360, 220){\line(-1, 0){80}}

\put(160, 140){\line(0, 1){80}}
\put(240, 220){\line(0, -1){80}}
\put(160, 140){\line(1, 0){80}}
\put(240, 220){\line(-1, 0){80}}

\put(40, 140){\line(0, 1){80}}
\put(120, 220){\line(0, -1){80}}
\put(40, 140){\line(1, 0){80}}
\put(120, 220){\line(-1, 0){80}}

\put(40, 20){\line(0, 1){80}}
\put(120, 100){\line(0, -1){80}}
\put(40, 20){\line(1, 0){80}}
\put(120, 100){\line(-1, 0){80}}

\put(320, 150){\line(0, 1){60}}
\put(320,160){\circle*{3}}
\put(320,170){\circle*{3}}
\put(320,180){\circle*{3}}
\put(320,200){\circle*{3}}

\put(305,160){$y_5$}
\put(305,170){$y_4$}
\put(305,180){$y_3$}
\put(305,200){$y_1$}

\put(203,203){$y_2$}

\put(245,185){bl-up}
\put(245,170){$y_1,y_3,$}
\put(245,160){$y_4,y_5$}

\put(125,185){bl-up}
\put(130,170){$y_2$}

\put(35,120){contract}
\put(85,125){$F_1:-4$-curve}
\put(85,115){$F_2:-2$-curve}

\put(200, 150){\line(0, 1){60}}
\put(200,200){\circle*{3}}
\put(180, 200){\line(1, 0){40}}
\put(180, 180){\line(1, 0){40}}
\put(180, 170){\line(1, 0){40}}
\put(180, 160){\line(1, 0){40}}

\put(80, 150){\line(0, 1){60}}
\put(60, 200){\line(1, 0){50}}
\put(60, 180){\line(1, 0){40}}
\put(60, 170){\line(1, 0){40}}
\put(60, 160){\line(1, 0){40}}
\put(90, 210){\line(1, -1){20}}

\put(60, 80){\line(1, -1){40}}
\put(60, 60){\line(1, 0){40}}
\put(60, 40){\line(1, 1){40}}
\put(80, 40){\line(0, 1){40}}

\put(80,60){\circle*{3}}
\put(90,70){\circle{3}}
\put(130,57){\circle*{3}}
\put(130,73){\circle{3}}
\put(135,70){$\cdots A_1$-sing}
\put(135,55){$\cdots \frac{1}{4}(1,1)$-sing}

\put(75,210){$-4$}
\put(45, 198){$-1$}
\put(45, 178){$-1$}
\put(45, 168){$-1$}
\put(45, 158){$-1$}
\put(105, 183){$-2$}

\put(60,225){$Y(2,1,1,1)$}
\put(320,225){$\mathbb{P}^2$}
\put(60,10){$Z(2,1,1,1)$}

\end{picture}
\caption{Construction of $Z(2,1,1,1)$}\label{ldpzu}
\end{figure}
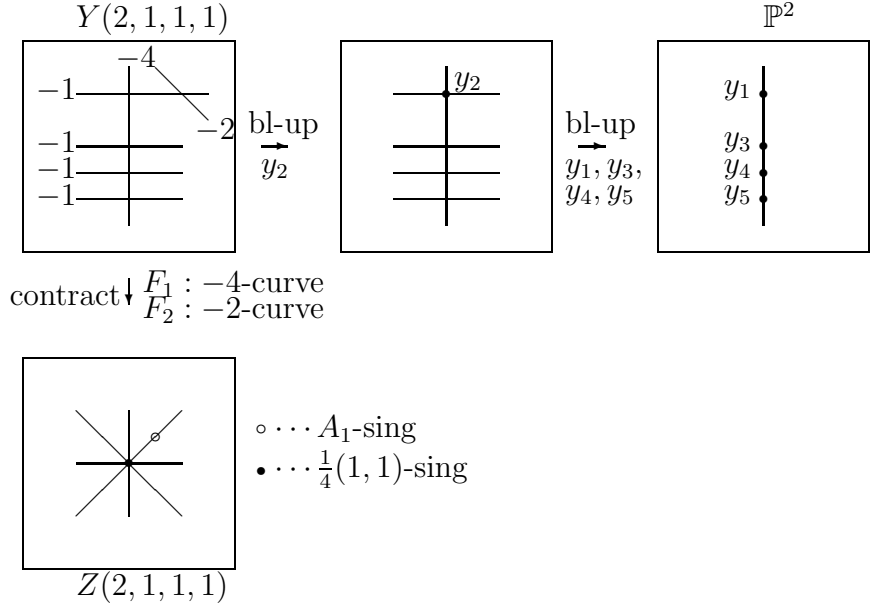

$Z(k_1,\ldots,k_m)$ is the log del Pezzo surface constructed as follows: 
let $z_1, \ldots ,z_m \in l \subset \mathbb{P}^2$ be distinct points on a line $l$.  
 Let 
 \[
 Y(k_1,\ldots,k_m)=Y_6 \stackrel{\phi_5}{\rightarrow} Y_5 \rightarrow 
 \cdots  \stackrel{\phi_1}{\rightarrow} Y_1 = \mathbb{P}^2
 \]  
 be the composition of blow-ups $\phi_i$ at some $y_i \in Y_i$ which we define as follows. 
 We can write $i = k_1+ \cdots +k_{m'}+k(i)$ for some 
 $0 \le m'\le m-1$ 
 and some $0 < k(i) \le k_{m'+1}$. Then $y_i$ is the strict transform of $z_{m'+1}$ if $k(i) = 1$, and 
 $y_i = \tilde{l} \cap \phi_{i-1}^{-1}(y_{i-1})$ otherwise, where $\tilde{l}$ is the strict transform of the line $l$. 
 By contracting a $-4$-curve and $-2$-curves on $Y(k_1,\ldots,k_m)$, we get a birational morphism $Y(k_1,\ldots,k_m) \rightarrow Z(k_1,\ldots,k_m)$. Thus we get a log del Pezzo surface $Z(k_1,\ldots,k_m)$ with only T-singularities.
We write $Z_1:= Z(1,1,1,1,1), Z_2 := Z(2,1,1,1), \ldots, Z_7 := Z(5)$. These are $Z_i (i=1,\ldots7)$ in Theorem \ref{ldpclassif}. 
\end{rem}

\section{Proof of Theorem \ref{ldpclassif}}\label{ldp}
In this section, we prove Theorem \ref{ldpclassif}.

\begin{proof} [Proof of Theorem \ref{ldpclassif} ]
Let $\nu: Y \rightarrow X$ be the minimal resolution of $X$. 
We can write
\[
K_Y = \nu^{\ast} K_X - \sum_{i=1}^k a_i F_i \ \ 
\] 
for some rational number $a_i$ such that $0 \le a_i <1$, 
where $\nu^{-1}(\Sing X) = \bigcup_{i=1}^k F_i$ is the irreducible decomposition. Then $F_i \simeq \mathbb{P}^1, 
F_i^2 \le -2$ since $K_Y$ is $\nu$-nef and $X$ has only quotient singularities. 

 We have $\varepsilon_{\gen}(-K_X) = \varepsilon_{\gen}(\nu^{*}(-K_X)) = \varepsilon_{\gen}(-K_Y - \sum_{i=1}^k a_i F_i) 
= \varepsilon(-K_Y - \sum_{i=1}^k a_i F_i,y) $, where $y \in Y$ is in very general position.

\vspace{5mm}

\noindent(I) First, we consider the case $K_X^2 = 4,5,6,7,8$. 
Suppose $X$ has only canonical singularities. Since $a_i=0$ for all $i$ and $Y$ is weak del Pezzo, we have  
$\varepsilon_{\gen}(-K_X)= \varepsilon_{\gen}(-K_Y)=2$ by Corollary \ref{same}. 

Suppose $X$ has a non-canonical singularity. Then there exists some $F_{J}$ such that $F_{J}^2 \le -3$.  
Note that  $a_j > 0$ if $F_j^2 \le -3$. 

\vspace{5mm}

First, suppose that $Y$ has a birational morphism $\chi:Y \rightarrow \mathbb{F}_n$ 
for some $n \ge 3$. 
We can assume that $F_1$ is the strict transform of the negative section $X_0 \subset \mathbb{F}_n$ 
and so $a_1 >0$. 
Then $\varepsilon_{\gen}(-K_Y -\sum_{i=1}^k a_i F_i) \le (-K_Y-\sum_{i=1}^k a_i F_i) \cdot \mu^{\ast} f \le 2 -  a_1 \mu^{\ast}X_0 \cdot \mu^{\ast} f = 2- a_1 < 2$, where $f \in \Pic \mathbb{F}_n$ is the class of a fiber of $\mathbb{F}_n \rightarrow \mathbb{P}^1$. So we are done. 

\vspace{5mm}
 
In the rest of the cases, we see that  $Y$ has a birational morphism to $\mathbb{P}^2$ as follows 
\[
\chi: Y=Y_{r+1} \stackrel{\rho_r}{\rightarrow} Y_r \stackrel{\rho_{r-1}}{\rightarrow} \cdots Y_2 \stackrel{\rho_1}{\rightarrow} Y_1 = \mathbb{P}^2,  
\] 
where $\rho_i$ is the blow up at $y_i \in Y_i$.  Set $\chi_i := \rho_i \circ \cdots \circ \rho_r \colon Y_{r+1} \rightarrow Y_i$ and 
 $E_i = (\rho_{i+1} \circ \cdots \circ \rho_r)^* \rho_i^{-1}(y_i)$ for $i=1,\ldots,r$. Note that $E_i$ is a $-1$-class.   
We fix the following isomorphism 
\[
\Phi \colon \Pic Y \simeq \mathbb{Z}^{r+1} ; \chi^* \mathcal{O}_{\mathbb{P}^2}(1) \mapsto (1;0^{r}), E_i \mapsto (0;0,\ldots,\stackrel{i}{-1},0, \ldots,0). 
\]
Define $\alpha^{(j)}, \beta^{(j)}_i \in \mathbb{Z}$ by 
$F_j = \chi^{\ast}\mathcal{O}_{\mathbb{P}^2}(\alpha^{(j)}) - \sum_{i=1}^r \beta^{(j)}_i E_i$ for $j=1,\ldots,k$. Then $F_j$ 
corresponds to $(\alpha^{(j)}; \beta_1^{(j)}, \ldots, \beta_r^{(j)})$.

We treat the following 3 cases (i),(ii),(iii) corresponding to the degree $\alpha^{(J)}$ of the negative curve $F_{J}$.   

\vspace{5mm}

\noindent (i) Assume that there exists $F_{J}$ such that $F_{J}^2 \le -3$ and $\alpha^{(J)} \ge 2$.

\begin{claim} 
We have $\alpha^{(J)} > \beta^{(J)}_i$ for all $i$.  
\end{claim}

\begin{proof}[Proof of Claim]
We see that   
\[
0 < (\chi^{\ast} \mathcal{O}_{\mathbb{P}^2}(1) -E_i) \cdot F_{J} = \alpha^{(J)} - \beta^{(J)}_i.   
\]  
Indeed $|\chi^{\ast} \mathcal{O}_{\mathbb{P}^2}(1) -E_i|$ contains an effective divisor which is a sum of the strict transform $\tilde{l}$ of a line on $\mathbb{P}^2$ and 
several $\chi$-exceptional curves. 
We can assume that the line intersects $\chi(F_{J})$ at some point which is not the blow-up centre of $\chi$. Hence we have $\tilde{l} \cdot F_{J} >0$. 
Since $F_{J}$ is irreducible and not $\chi$-exceptional, we get the claim.   
\end{proof}

Let ${l}_1(y) \subset \mathbb{P}^2$ be the line through $\chi(y),y_1 \in \mathbb{P}^2$ and 
$\tilde{l}_1(y) \subset Y$ be its strict transform.  
Note that, if $i \neq 1$, then  
$\tilde{l}_1(y) \cap E_i = \emptyset$ 
since $\chi( y) $ is a very general point on $\mathbb{P}^2$. So 
$\tilde{l}_1(y) \in |\chi^{\ast}\mathcal{O}_{\mathbb{P}^2}(1)- E_1|$. 
Also note that $\tilde{l}_1(y)$ is irreducible and 
$\tilde{l}_1(y)^2 =0$, hence it is nef. 
Then, since $F_{J} \cdot \tilde{l}_1(y) = \alpha^{(J)}- \beta^{(J)}_1 > 0$ and $a_{J} > 0$,  
we get 
\begin{equation}\label{sameineq}
(-K_Y - \sum_{i=1}^k a_i F_i) \cdot \tilde{l}_1(y) = 2 - \sum_{i=1}^k a_i F_i \cdot \tilde{l}_1(y) < 2,  
\end{equation}
and so $ \varepsilon_{\gen}(-K_X) <2$. 

\vspace{5mm}

\noindent(ii) 
Next, assume that there exists $F_{J}$ such that $F_{J}^2 \le -3$ and $\alpha^{(J)} = 0$. 
Put $\chi'_i:= \rho_1 \circ \cdots \circ \rho_{i-1} : Y_i \rightarrow  \mathbb{P}^2$. 
Let $E'_j \subset Y$ be the strict transform of a $-1$-curve $\rho_j^{-1}(y_j) \subset Y_{j+1}$ for $j=1,\ldots,r$.
We can assume that $y_1= \chi(F_{J})$ and $y_2, \ldots, y_s$ for some positive integer $s$ are all the points over $y_1$, that is, $y_1 = \chi'_2(y_2) = \cdots = \chi'_s(y_s)$. Note that  
\[
\chi^{-1}(y_1) = \bigcup_{i=1}^{s} E'_i. 
\]   
Moreover there exist some integers $J(1), \ldots, J(l)$ such that $J(1) =1$, $E'_{J(l)} = F_J$ and, for distinct integers $i, k$, we have  
\begin{equation}\label{relE'}
E'_{J(i)} \cdot E'_{J(k)} = 
\begin{cases} 
 1 & i-k = 1  \\ 
 0 & i-k >1 .  
\end{cases}
\end{equation} 
If ${E'_{1}}^2 \le -3$, we can assume that $l=1$. 
If ${E'_{1}}^2 \ge -2$, we see that ${E'_{1}}^2 = -2$ and can assume that $(E'_{J(i)})^2 \ge -2$ for $i <l$. 

\begin{claim}\label{-2claim}
If ${E'_1}^2 = -2$ and we use the above assumptions, we see that $(E'_{J(i)})^2 = -2$ for $i=1,\ldots,l-1$. 
\end{claim} 

\begin{proof}[Proof of Claim]
If not, there exists some $l'$ such that $1<l'<l$ and $(E'_{J(l')})^2 = -1$ and 
$(E'_{J(i)})^2 =-2$ for  $i =1,\ldots l'-1$. 
We can write the class of $E'_1=E'_{J(1)}$ as 
\[
E'_{J(1)} = E_{J(1)} -E_{J(2)} 
\] 
which corresponds to $(0;-1,1,0^{r-2}) \in \mathbb{Z}^{r+1}$ since $y_2 \in \rho_1^{-1}(y_1)$. 
By the relations (\ref{relE'}), we have 
$E'_{J(2)} = E_{J(2)} - E_{J(3)}, \ldots, E'_{J(l'-1)} = E_{J(l'-1)}- E_{J(l')}$ in order. By the same relation, 
we also have  
$E'_{J(l'+1)} \le E_{J(l'+1)} - \sum_{i=1}^{l'} E_{J(i)} $
and $(E'_{J(l'+1)})^2 \le -3$. Hence we can write $\Phi(E'_{J(l'+1)})$ as $(0;1^{l'},-1,*,\ldots,*) \in \mathbb{Z}^{r+1}$. 
However this can not be effective since $E_1=E_{J(1)}$ is the exceptional divisor of the blow-up of a point on $\mathbb{P}^2$. 
    
\end{proof}

Thus we can assume that $E'_{J(i)}= F_i$ for $i=1,\ldots, l$ by changing the orders of $F_j$. 

\begin{claim}
Under the above assumptions, we have 
$a_{1} >0$. 
\end{claim} 

\begin{proof}[Proof of Claim] 
If $l=1$, we are done. 
If $l>1$, by Claim \ref{-2claim}, we have 
\[
0 = K_Y \cdot E'_{J(l-1)} = -\sum_{j=1}^k a_j F_j \cdot E'_{J(l-1)} \le -a_{l-1} F_{l-1}^2 - a_{J} F_{J} \cdot F_{l-1} <  -a_{l-1} F_{l-1}^2 = 2 a_{l-1} 
\]
and get $a_{l-1}>0$. We can continue this process to get $a_{1} >0$. 
\end{proof}

Let $l_{1}(y)\subset \mathbb{P}^2$ be the line through $\chi(y)$ and $y_1$ and 
$\tilde{l}_{1}(y) \subset Y$ be its strict transform. 
Then, by an argument which is similar to that in (i), we get
the inequality (\ref{sameineq}) again,  
and so $\varepsilon_{\gen}(-K_X) < 2$. 

\vspace{5mm}

\noindent(iii) 
In the rest of the cases, all $F_j$ such  that $F_j^2 \le -3$ satisfies that $\alpha^{(j)}=1$. 
We can assume $F_1^2 \le -3$ without loss of generality.

Suppose that $F_1$ satisfies $\beta^{(1)}_i = 1$ for all $i$. Then $F_j$ for $j \neq 1$ satisfies that $\alpha^{(j)}=0$. Indeed, if there is  
$F_{j'}$ such that $\alpha^{(j')}\ge 1$, then we have $-K_Y \cdot F_{j'} = 3 \alpha^{(j')} - \sum_{i=1}^r \beta^{(j')}_i \le 0$ and 
\[
F_1 \cdot F_{j'} = \alpha^{(j')} - \sum_{i=1}^r \beta^{(j')}_i <0 
\] 
since $\alpha^{(j')}>0$ and this is a contradiction. 
Thus we get $\alpha^{(j)}=0$ for $j \neq 1$ and  $F_ j^2 =-2$ for all $j \neq 1$ by the assumption of (iii). 
We also see that $F_1 \cdot F_j =0$ for $j \neq 1$ by the above argument. 
Hence we get $K_Y = \nu^* K_X -a_1 F_1$. We see that $-2-F_1^2 = K_Y \cdot F_1 = -a_1 F_1^2$ and 
\[
a_1 = 1 + \frac{2}{F_1^2}
\] 
and 
\[
\mathbb{Z} \ni K_X^2 = K_Y^2 + a_1^2 F_1^2 = K_Y^2 + F_1^2 +4 + \frac{4}{F_1^2}.  
\]
Hence $F_1^2 = -4$ and $r=5$. Then $Y$ can be constructed by blowing up a point 5 times on a line on 
$\mathbb{P}^2$. 
 Therefore, we can see that $Y$ is one of $Y(k_1,\ldots,k_m)$ and $X$ is one of $Z(k_1,\ldots,k_m)$. 
In this case, we can compute $\varepsilon_{\gen}(-K_X) =2$. Indeed, we can see that 
\[
\Phi(\nu^* ({-}K_X)) = \Phi(-K_Y - \frac{1}{2}F_1) = (3;1^5) - \frac{1}{2}(1;1^5) = \left(\frac{5}{2}; \left( \frac{1}{2} \right)^5 \right), 
\]
\[
\varepsilon_{\gen}(X,-K_X) = \varepsilon_{\gen}\left( Y,\frac{1}{2} 
\left( \chi^{\ast} \mathcal{O}_{\mathbb{P}^2}(5) - \sum_{i=1}^5 E_i \right) \right) = 2  
\]
by calculating as in Example \ref{compex5}. 

In the rest case, there exists some $i$ such that $\beta^{(1)}_i =0$. We can assume that $y_1 = \chi'_i(y_i)$. 
\begin{claim}
 If $\chi'_i(y_i) \in \chi(F_1)$,  
then $E'_1$ is a $(-2)$-curve and $F_1 \cap E'_1 \neq \emptyset$. 
\end{claim} 
\begin{proof}[Proof of Claim]
By the assumption of (iii), an irreducible $\chi$-exceptional curve is a $-1$-curve or a $-2$-curve. If $E'_1$ is a $-1$-curve, we see that $i=1$ and thus $\beta_i^{(1)}= \beta_1^{(1)}=1$.
 It contradicts the assumption. 
Hence we see that $E'_1$ is a $-2$-curve.  If $E'_1 \cap F_1 = \emptyset$, the $\chi$-exceptional curves over $\chi'_i(y_i)$ 
form a tree and the curve which intersects with $F_1$ is a $-1$-curve. 
This only happens if $\chi$-exceptional curves over $\chi'_i(y_i)$ are generated by blowing up intersection points of the strict transforms of $\chi(F_1)$ and the exceptional curves and 
we see that $\beta^{(1)}_i =1$. 
This contradicts that $\beta^{(1)}_i =0$. Thus we see that $E'_1 \cap F_1 \neq \emptyset$. 
 \end{proof}

 Hence we can write $F_2= E'_1$ and see that $a_2>0$ if $\chi'_i(y_i) \in \chi(F_1)$. 
In both cases, we take a line $l_i(y)$ on $\mathbb{P}^2$ which passes through
$\chi(y)$ and $\chi'_i (y_i)$. Let $\tilde{l}_i(y)$ be the strict transform of $l_i(y)$. Then  
the inequality (\ref{sameineq}) holds again.  
 Indeed, if $\chi'_i(y_i) \in \chi (F_1)$, we have $a_1>0, \tilde{l}_i(y) \cdot F_{2} >0 $ and this implies the inequality (\ref{sameineq}). 
If $\chi'_i(y_i) \notin \chi(F_1)$, we have $a_1 >0, \tilde{l}_i(y) \cdot F_1 >0$ and this implies the inequality (\ref{sameineq}).   
And so $\varepsilon_{\gen}(-K_X) < 2$. 

\vspace{5mm}

\noindent (II) Next, consider the case $K_X^2=9$ and $X$ is not smooth. We can see that $X$ has a non-canonical singularity 
since a weak del Pezzo surface of degree $9$ is $\mathbb{P}^2$. 
Then  we see that $Y$ has a birational morphism to $\mathbb{P}^2$ or to $\mathbb{F}_n$ for some $n \ge 3$. 

If $Y$ has a birational morphism to $\mathbb{P}^2$, the morphism is non-trivial and we take a line $l_1(y)$ through $\chi'_1(y_1)$ and $\chi(y)$. Let $\tilde{l}_1(y) \subset Y$ be its strict transform. 
Then, as in (I), we have 
\[
\varepsilon_{\gen}(-K_X) = \varepsilon_{\gen}(-K_Y - \sum_{j=1}^k a_j F_j) \le (-K_Y- \sum_{j=1}^k a_j F_j) \cdot \tilde{l}_1(y) \le 2. 
\]

If $Y$ has a birational morphism to $\mathbb{F}_n$ for some $n \ge 3$, then we see that $\varepsilon_{\gen}(-K_X) \le 2$ by considering the strict transform of the fiber class as in Case 2 of (I). 
\end{proof}

\section*{Acknowledgment}
This note is the expanded version of the master's thesis of the author at the University of Tokyo. 
The author is grateful to his advisor Professor Yujiro Kawamata for his warm encouragement and valuable comments. 
He would like to thank Professors Brian Harbourne, Sandra Di Rocco for valuable comments and suggestions. 
He would like to thank Professor Miles Reid and Doctors Michael Selig and Thomas Ducat 
for checking the manuscript carefully and improving the presentations. 
He is grateful to the anonymous referee for many constructive comments about the presentations of the paper. 
He is also grateful to Doctor Atsushi Ito for checking the manuscript carefully and many useful comments. 
The author was partially supported by the Global COE program of the University of Tokyo and 
Warwick Postgraduate Research Scholarship.


\begin{thebibliography}{Kol}

\bibitem{B1}
T.~Bauer, 
\emph{Seshadri constants of quartic surfaces},
 Math. Ann. 309 (1997), no. 3, 475--481.

\bibitem{B2}
T.~Bauer, 
\emph{Seshadri constants on algebraic surfaces}, 
Math. Ann. 313 (1999), no. 3, 547--583.

\bibitem{primer}
T.~Bauer, S.~Di Rocco, B.~ Harbourne, M.~Kapustka, A.~Knutsen, W.~Syzdek, T.~Szemberg, 
\emph{A primer on Seshadri constants},
 Interactions of classical and numerical algebraic geometry, 33--70, Contemp. Math., 496, Amer. Math. Soc., Providence, RI, 2009.


\bibitem{BSch}
T.~Bauer, C.~Schulz, 
\emph{Seshadri constants on the self-product of an elliptic curve},
 J. Algebra 320 (2008), no. 7, 2981--3005.

\bibitem{BS}
T.~Bauer, T.~Szemberg, 
\emph{Local positivity of principally polarized abelian threefolds},  
J. Reine Angew. Math. 531 (2001), 191--200. 


\bibitem{Br}
A.~Broustet,
\emph{Constantes de Seshadri du diviseur anticanonique des surfaces de del Pezzo}, Enseign. Math. (2) 52 (2006), no. 3-4, 231--238. 

\bibitem{Dem}
J.P.~Demailly, 
\emph{Singular Hermitian metrics on positive line bundles}, Complex algebraic varieties (Bayreuth, 1990), 87--104, 
Lecture Notes in Math., 1507, Springer, Berlin, 1992. 

\bibitem{G}
L.F.~Garc\'{i}a, 
\emph{Seshadri constants on ruled surfaces: the rational and the elliptic cases},  
Manuscripta Math. 119 (2006), no. 4, 483--505. 

\bibitem{HP}
P.~Hacking, Y.~Prokhorov, 
\emph{Smoothable del Pezzo surfaces with quotient singularities}, 
Compos. Math. 146 (2010), no. 1, 169--192. 

\bibitem{LH}
M.~Lahyane, B.Harbourne, 
\emph{Irreducibility of $-1$-classes on anticanonical rational surfaces and finite generation of the effective monoid}, 
Pacific J. Math. 218 (2005), no. 1, 101--114. 


\bibitem{Laz}
R.~Lazarsfeld, 
\emph{Positivity in algebraic geometry I}, 
 Ergebnisse der Mathematik und ihrer Grenzgebiete. 3. Folge. A Series of Modern Surveys in Mathematics, 
 48. Springer-Verlag, Berlin, 2004. xviii+387 pp.


\bibitem{N}
M.~Nakamaye, 
\emph{Seshadri constants on abelian varieties}, 
 Amer. J. Math. 118 (1996), no. 3, 621--635.

\end{thebibliography}
\end{document}